\documentclass[11pt]{article}
\usepackage[margin=1in]{geometry}
\usepackage[T1]{fontenc}
\usepackage{amsmath,amssymb,amsthm,comment}
\usepackage[hidelinks]{hyperref}
\numberwithin{equation}{section}
\newtheorem{theorem}{Theorem}[section]
\newtheorem{conjecture}[theorem]{Conjecture}
\newtheorem{proposition}[theorem]{Proposition}
\newtheorem{lemma}[theorem]{Lemma}

\theoremstyle{definition}
\newtheorem{definition}[theorem]{Definition}
\newtheorem{remark}[theorem]{Remark}
\newcommand{\C}{\mathbb C}
\newcommand{\B}{\mathbb B}
\newcommand{\PP}{\mathbb P}
\newcommand{\cO}{\mathcal O_{\mathrm{reg}}}
\newcommand{\ii}{\sqrt{-1}}
\newcommand{\ord}{\operatorname{ord}}
\newcommand{\rank}{\operatorname{rank}}
\newcommand{\Reg}{\operatorname{Reg}}
\newcommand{\pr}{\operatorname{pr}}
\newcommand{\Qnn}{\mathbb Q_{\geq0}}

\title{Componentwise rigidity of holomorphic isometric maps between bounded symmetric domains\\with negative conformal factors}
\author{Ming Xiao\thanks{The author is supported in part by the NSF grants DMS-2045104 and DMS-2554635.}}

\date{}

\begin{document}
\maketitle

\begin{abstract}
Let $D$ be an irreducible bounded symmetric domain of complex dimension at least two. We study local holomorphic maps from $D$ into products of irreducible bounded symmetric domains satisfying a metric identity with positive and negative conformal factors. Under a noncancellation condition on these factors, we prove that every nonconstant component extends to a holomorphic isometric embedding up to a positive integer factor. This proves a componentwise rigidity conjecture formulated by Cheng--Hao--Yuan--Zhang. 
\end{abstract}



\section{Introduction}\label{sec:introduction}

The study of holomorphic isometric maps goes back to the classical work of Calabi \cite{Ca}. Holomorphic isometries between bounded symmetric domains attracted particular attention due to their connection with arithmetic geometry, especially with modular correspondences; see Clozel--Ullmo \cite{CU}. Mok \cite{Mok2002,Mok2011,Mok2012,Mok2016} initiated a systematic study
of local holomorphic isometric maps between bounded symmetric domains, which has stimulated extensive further research.

Let $D_1$ and $D_2$ be bounded symmetric domains, and write $ds_D^2$ for the Bergman metric of a bounded symmetric domain $D$. Suppose that $F:V\to D_2$ is a holomorphic map from an open connected subset $V\subseteq D_1$ satisfying
$$
F^*(ds_{D_2}^2)=\lambda ds_{D_1}^2 \quad\hbox{on }V,
$$
for some positive constant $\lambda$. In his groundbreaking work \cite{Mok2012}, Mok proved that $F$ is algebraic and extends to a holomorphic proper and isometric embedding from $D_1$ into $D_2$. In the same paper, Mok also proved that $F$ must be totally geodesic when $D_1$ is irreducible and has rank at least two.

Much less is known when the source domain $D_1$ has rank one, namely when it is a complex unit ball $\B^n$. In this case, the rigidity problem has a very different character. In the one-dimensional case (i.e., when $n=1$), the problem of holomorphic isometric maps from the Poincar\'e disk
into polydisks has been intensively studied by many authors, beginning with the work of Clozel--Ullmo \cite{CU} and Mok \cite{Mok2012}. In particular, Mok \cite{Mok2012} constructed non-totally geodesic holomorphic isometric maps from the Poincar\'e disk into polydisks, nowadays
called Mok's $p$-th root maps. For further investigations of this problem and related questions, we refer the readers to Ng \cite{Ng1}, Chan \cite{Ch1,Ch2}, Chan--Yuan \cite{CY}, Chan--Xiao--Yuan \cite{CXY2017}, and references therein. When $n\geq2$, the problem of holomorphic isometric maps from $\B^n$ into products of complex unit balls was studied by Mok \cite{Mok2002, Mok2012} and Ng \cite{Ng2}, and subsequently by Yuan--Zhang \cite{YZ}. Let
$F=(F_1,\ldots,F_m)$ be a holomorphic map from an open connected subset $V \subseteq \B^n$ into $\B^{N_1}\times\cdots\times\B^{N_m}$ satisfying
$$
ds_{\B^n}^2=\sum_{i=1}^m\lambda_iF_i^*(ds_{\B^{N_i}}^2)
\quad\text{on }V,
$$
where the $\lambda_i$ are positive constants. Yuan--Zhang \cite{YZ} proved that every nonconstant component $F_i$ extends to a totally geodesic map from $\B^n$ into $\B^{N_i}$.

On the other hand, when $D_1$ is the ball and $D_2$ has an irreducible factor of rank at least two, rigidity in the form of total geodesy fails dramatically. Indeed, Mok \cite{Mok2016} constructed non-totally geodesic holomorphic isometric embeddings from complex
unit balls into higher-rank irreducible bounded symmetric domains of sufficiently large dimension. See also subsequent developments in the works of Chan--Mok \cite{CM}, Upmeier--Wang--Zhang \cite{UWZ}, Xiao--Yuan \cite{XY1,XY2}, and Chan \cite{Ch3}, among others.

Although total geodesy fails for general higher-rank targets, it has been believed that a weaker rigidity property should still hold when the target is reducible. Yuan formulated this conjecture explicitly in \cite[Problem~5.2]{Yuan2019} in the setting of a ball source domain and
positive conformal factors, which motivated subsequent research. Partial results were obtained in \cite{YZ} and \cite{Xiao2022}. Recently, Yuan's conjecture was resolved in full generality, independently by Yuan--Zhang \cite{YZ2026} and by the author \cite{Xiao2026}.

Another closely related line of research that has received extensive attention concerns holomorphic isometric identities allowing negative conformal factors. In their pioneering work, Huang--Yuan \cite{HY} established rigidity results in this setting for products of compact Hermitian symmetric
spaces under a numerical condition that rules out cancellation between the positive and negative terms. Cheng--Hao--Yuan--Zhang \cite{CHYZ} considered the corresponding problem for maps from $\B^n$, $n\geq2$, into products of complex unit balls. Under the same noncancellation condition, they proved that every nonconstant component extends to a totally geodesic holomorphic isometric embedding. They
also posed a rigidity problem for general bounded symmetric domains \cite[Problem~5.1]{CHYZ}, which generalizes Yuan's original conjecture \cite[Problem~5.2]{Yuan2019}.

To state the conjecture in \cite{CHYZ}, we first fix the notation and metric normalization.
In this paper, all bounded symmetric domains are taken in their Harish--Chandra realizations. For an irreducible bounded symmetric domain $D$, let $g_D$ be its canonical K\"ahler--Einstein metric normalized so that minimal disks have Gaussian curvature $-2$, and let $\omega_D$ be its K\"ahler form. Let $\Qnn$ denote the set of nonnegative rational numbers. For positive real numbers $c_1,\cdots,c_q$, write
$$
\operatorname{span}_{\Qnn}\{c_1,\cdots,c_q\}
:=\left\{\sum_{\alpha=1}^q r_\alpha c_\alpha:r_\alpha\in\Qnn\right\}.
$$
The span of an empty family is understood to be $\{0\}$.

\begin{conjecture}[{Cheng--Hao--Yuan--Zhang, \cite[Problem~5.1]{CHYZ}}]\label{conj1}
Let $D\subseteq\C^n$, $n\geq2$, be an irreducible bounded symmetric domain, and let $V \subseteq D$ be a nonempty connected open set. Let $\Omega_i\subseteq\C^{N_i}$, $1\leq i\leq m$, and $\Lambda_j\subseteq\C^{M_j}$, $1\leq j\leq v$, be irreducible bounded symmetric domains, where $m\geq1$ and $v\geq0$. Suppose that $F_i:V\longrightarrow\Omega_i$ and $G_j:V\longrightarrow\Lambda_j$ are nonconstant holomorphic maps satisfying
	\begin{equation}\label{eq:metric-identity}
		g_D=\sum_{i=1}^m\lambda_i F_i^*g_{\Omega_i}
		-\sum_{j=1}^v\mu_j G_j^*g_{\Lambda_j}\quad\hbox{on }V,
	\end{equation}
for positive constants $\lambda_1,\ldots,\lambda_m,\mu_1,\ldots,\mu_v$. Assume
	\begin{equation}\label{eq:noncancellation}
		\operatorname{span}_{\Qnn}\{\lambda_1,\cdots,\lambda_m\}
		\cap\operatorname{span}_{\Qnn}\{\mu_1,\cdots,\mu_v\}=\{0\}.
	\end{equation}
Then every $F_i$ and $G_j$ extends to a proper holomorphic embedding of $D$ into its target domain, still denoted by the same symbol, such that
	\begin{equation}\label{eq:component-isometry}
		F_i^*g_{\Omega_i}=k_i g_D,\qquad G_j^*g_{\Lambda_j}=\ell_j g_D
		\quad\hbox{on }D,
	\end{equation}
for some positive integers $k_i$ and $\ell_j$. Moreover,
	\begin{equation}\label{eq:sum-factors}
		\sum_{i=1}^m\lambda_i k_i-\sum_{j=1}^v\mu_j\ell_j=1.
	\end{equation}
\end{conjecture}

\begin{remark}
Our formulation above differs superficially from that of \cite[Problem~5.1]{CHYZ}: we use normalized K\"ahler--Einstein metrics, whereas \cite{CHYZ} uses Bergman metrics. But for an irreducible bounded
symmetric domain $E$, its Bergman metric is $\gamma_E g_E$, where $\gamma_E$ is its genus, which is a positive integer. Thus changing the metric normalization merely rescales the conformal factors by positive rational numbers and leaves the noncancellation condition \eqref{eq:noncancellation} unchanged. Consequently, the two formulations are indeed equivalent.
\end{remark}

\begin{remark}\label{rmk1d3}
The hypothesis \eqref{eq:noncancellation} cannot be omitted from Conjecture~\ref{conj1}, as shown by \cite[Example~1.4]{CHYZ}: For any nonconstant holomorphic map $H:V\to\B^N$ and any $c>0$, the maps $F_1(z)=z$ and $F_2=G_1=H$ satisfy
$$
g_D=F_1^*g_D+cF_2^*g_{\B^N}-cG_1^*g_{\B^N},
$$
while
$$
c\in\operatorname{span}_{\Qnn}\{1,c\}\cap\operatorname{span}_{\Qnn}\{c\}.
$$
Thus \eqref{eq:noncancellation} fails, and $H$ may be chosen to be neither isometric nor  proper.
	
The assumption $n\geq2$ is also essential in Conjecture~\ref{conj1}, as Mok's $p$-th root maps from the Poincar\'e disk into polydisks \cite{Mok2012} serve as counterexamples when $n=1$.
\end{remark}

When $v=0$ (so that no $G_j$ or $\mu_j$ appear), Conjecture~\ref{conj1} reduces to the case of positive conformal factors. In this case,  if the source domain $D$ has rank at least two, the conjecture already follows from Mok's rigidity theorem \cite[Theorem~1.3.2(a)]{Mok2012} and its proof. If $v=0$ and $D$ is the unit ball, it reduces to Yuan's original conjecture \cite[Problem~5.2]{Yuan2019}, which was recently settled by \cite{YZ2026} and \cite{Xiao2026}. However, when $v>0$, negative conformal factors arise. Much less is known in this general signed setting, and new arguments are required. The purpose of the present paper is to give an affirmative answer to Conjecture~\ref{conj1} in its full generality.

\begin{theorem}\label{thm:main}
Conjecture~\ref{conj1} holds.
\end{theorem}

We discuss the main additional difficulties in the general setting compared with the cases of positive conformal factors and a ball source. The first difficulty is to establish algebraicity of the maps $F_i$ and $G_j$, as already pointed out in \cite{CHYZ}. In the positive-factor setting, algebraicity follows from Mok \cite[Theorem~2.1.2]{Mok2012} and Chan--Xiao--Yuan \cite[Theorem~4.25]{CXY2017}. In the signed setting, a separate argument is needed. We lift the problem to the Zariski closure $\mathcal V$ of the graph of the combined map $(F_1,\cdots,F_m,G_1,\cdots,G_v)$. Since algebraicity is not known a priori, the projection $\mathcal V\to\C^n$ may have positive-dimensional fibers over
a general point of $\C^n$. The first key step of our proof is to use condition \eqref{eq:noncancellation} to prove algebraicity.

The second difficulty arises when we replace the ball source by a general irreducible bounded symmetric
domain. For a ball source, the Segre varieties of the sphere
$\partial\B^n$ are affine hyperplanes. Their simple geometry allows one to prove that their lift $\Sigma$ to the algebraic covers is smooth, and then to establish its
irreducibility by using Bertini's irreducibility theorem. For a source domain $D$ of rank at least two, the associated Segre varieties may be singular. Although they extend to hyperplane sections of the compact dual of $D$, these hyperplanes form a lower-dimensional
family to which the preceding Bertini irreducibility argument does not apply. Another key part of our proof is to develop a new method to establish the irreducibility of $\Sigma$ in this setting, which is crucial for the subsequent argument. Our method fundamentally utilizes the Segre geometry of Hermitian symmetric spaces developed by Fang-Huang-Xiao \cite{FHX}.

Besides the works mentioned above, for more studies of metric-preserving mappings and mappings preserving invariant $(p,p)$-forms between bounded symmetric domains, we refer readers
to \cite{MN,Y1,Yang2017,CY2,Mok2026,DY} and the references therein. 
See also the survey papers \cite{Mok2011,Mok2018,Yuan2019} for accounts of related developments
at different stages. The paper is organized as follows. In \S\ref{sec:section1}, we present the preliminaries, prove algebraicity, and construct the algebraic covers. In \S\ref{sec:incidence}, we recall the required Segre geometry of Hermitian symmetric spaces,
prove irreducibility of $\Sigma$, and show that the constructed quotient functions are regular units on the algebraic cover. Finally, \S\ref{sec:product-unit} completes the proof of Theorem~\ref{thm:main}. We remark that the noncancellation condition \eqref{eq:noncancellation} is used at two places: in \S\ref{sec:section1} to prove algebraicity of the maps, and in \S\ref{sec:incidence} to show that certain quotient functions are regular units.

\medskip

{\bf Acknowledgments.} The author thanks Zhiyuan Jiang and Junyi Xie for their patience in answering his questions about algebraic geometry, including some rather elementary ones, and for pointing him to relevant references during the preparation of this work. He especially thanks Zhiyuan for his help with the proof of Lemma~\ref{lem:projective-segre-pullback}. The author is also grateful to Yuan Yuan for his kind encouragement to develop the present paper to resolve \cite[Problem~5.1]{CHYZ}, as well as for his interest and valuable comments on the draft. Besides, some of the ideas appearing in \cite{YZ2026,Xiao2026} and this paper grew out of stimulating discussions with Yuan since their joint work \cite{XY1,XY2}. The author also thanks Xiaojun Huang and Hanlong Fang for helpful discussions on related topics, especially the Segre geometry of Hermitian symmetric spaces. The arXiv preprint \cite{Xiao2026} is an earlier version of this paper; we do not intend to submit that version separately for publication in a journal.

\medskip

{\bf Use of AI tools.}  OpenAI's ChatGPT assisted in simplifying the proof of Lemma~\ref{lem:RC} and subsequently pointed out that the lemma is a special case of Rosenlicht's product-unit theorem. (We keep the proof for completeness though.) The author takes full responsibility for the correctness and content of the manuscript.

\section{Preliminaries and algebraicity of the map}\label{sec:section1}

Let $(F_1,\cdots,F_m,G_1,\cdots,G_v)$ be as in the hypotheses of Conjecture~\ref{conj1}. Composing with automorphisms of the source domain and the target factors, we may assume that $0\in V$ and
\begin{equation}\label{eq:normalization}
F_i(0)=0,\qquad G_j(0)=0\quad\text{for all }i,j.
\end{equation}

\subsection{A differential form identity}\label{subsec:potentials}

For an irreducible bounded symmetric domain $E\subseteq\C_Z^d$, let
$\rho_E(Z,\overline Z)$ denote its generic norm. We recall its standard properties from
\cite{Loos1977} and \cite[Section~2.1]{CM}; see also the brief discussion in
\cite[Section~1]{Xiao2022}. With the usual normalization,
$\rho_E(Z,\overline Z)$ is a real-valued polynomial satisfying
\begin{equation}\label{eq:generic-norm-boundary}
	\rho_E(Z,\overline Z)>0
	\quad\text{for }Z\in E;\qquad
	\rho_E(Z,\overline Z)=0
	\quad\text{for }Z\in\partial E.
\end{equation}
Moreover, the normalized K\"ahler--Einstein metric $g_E$ on $E$ is induced from the K\"ahler form
\begin{equation}\label{eq:target-potential}
	\omega_E=-\ii\,\partial\bar\partial\log\rho_E(Z,\overline Z).
\end{equation}
Polarizing the generic norm, we obtain a polynomial $\rho_E(Z,\Xi)$ in the independent
complex variables $Z,\Xi\in\C^d$, which agrees with $\rho_E(Z,\overline Z)$
on the diagonal $\Xi=\overline Z$. By the property of the generic norm,
\begin{equation}\label{eq:generic-norm-axes}
	\rho_E(Z,0)=\rho_E(0,\Xi)=1.
\end{equation}
We write $\rho_i:=\rho_{\Omega_i}$ and $\sigma_j:=\rho_{\Lambda_j}$.

For any subset $O\subseteq\C^n$ (in particular, an open subset or a subvariety), define its conjugate set by $O^*:=\{\overline z:z\in O\}.$ For any complex-valued function $f$ on $O$, define its conjugate function by
$$
\overline f(\zeta):=\overline{f(\overline\zeta)},\qquad \zeta\in O^*.
$$
The same notation will be used componentwise for maps $F:O\to\C^N$ to define their conjugate maps $\overline F:O^*\to\C^N$. If $O$ is open and $F$ is holomorphic, then $\overline F$ is holomorphic on $O^*$. In particular, $\overline{F_i}$ and $\overline{G_j}$ are holomorphic on $V^*$.

Fix a small Euclidean ball $W_0$ centered at $0$ whose closure is contained in
$V$ (then $W_0^*=W_0$, but we retain the notation $W_0^*$ to distinguish
the two copies). We use independent complex coordinates $z$ on $W_0\subseteq\C_z^n$
and $\zeta$ on $W_0^*\subseteq\C_\zeta^n$. On $W_0\times W_0^*$, define
\begin{equation}\label{eq:Delta-R-S}
	\begin{aligned}
		\Delta(z,\zeta)&:=\rho_D(z,\zeta),\\
		R_i(z,\zeta)&:=\rho_i(F_i(z),\overline{F_i}(\zeta)),\\
		S_j(z,\zeta)&:=\sigma_j(G_j(z),\overline{G_j}(\zeta)).
	\end{aligned}
\end{equation}
Here $W_0$ and $W_0^*$ are sufficiently small such that these functions are close to $1$, so they have holomorphic logarithms with value zero at $(0,0)$.
Taking potentials in \eqref{eq:metric-identity} and polarizing yields
\begin{equation}\label{eq:polarized-log}
\log\Delta=\sum_{i=1}^m\lambda_i\log R_i-\sum_{j=1}^v\mu_j\log S_j, \quad \text{on}~~W_0\times W_0^*.
\end{equation}
Indeed, the difference of the two sides is the polarization of a pluriharmonic function, and hence is the sum of a holomorphic function of $z$ and a holomorphic function of $\zeta$. By \eqref{eq:normalization} and \eqref{eq:generic-norm-axes}, that difference vanishes on both coordinate
axes, so it is identically zero. Differentiating \eqref{eq:polarized-log} gives
\begin{equation}\label{eq:log-differential-base}
\frac{d\Delta}{\Delta}=\sum_{i=1}^m\lambda_i\frac{dR_i}{R_i}
-\sum_{j=1}^v\mu_j\frac{dS_j}{S_j}\quad \text{on}~~W_0\times W_0^*.
\end{equation}

\subsection{Lifting to the Zariski closure of the graph}\label{subsec:graph}

Let $N:=\sum_{i=1}^m N_i+\sum_{j=1}^v M_j$, and write $T=(F_1,\cdots,F_m,G_1,\cdots,G_v)$. We do not yet know that $T$ is algebraic. Its graph over $W_0$ is the complex submanifold
\begin{equation}\label{eq:graph}
	\Gamma:=\{(z,T(z)):z\in W_0\}\subseteq\C_z^n\times\C_w^N.
\end{equation}
Let $I(\Gamma)$ be the ideal of all polynomials $P(z,w)$ such that $P(z,T(z))$ vanishes identically on $W_0$. Set
\begin{equation}\label{eq:graph-closure}
\mathcal V:=\{(z,w)\in\C^{n+N}:P(z,w)=0\text{ for every }P\in I(\Gamma)\}.
\end{equation}
This is the Zariski closure of $\Gamma$, namely the smallest algebraic subset containing it. Its existence does not require algebraicity of $T$. Hilbert's basis theorem implies that finitely many equations suffice to define it. A priori, its dimension can exceed that of the graph: for example, the graph of $z\mapsto e^z$ on any disk has Zariski closure $\C^2$, in which case $I(\Gamma)=\{0\}$.

The variety $\mathcal V$ is irreducible. Indeed, if $P_1P_2\in I(\Gamma)$, then
$$
P_1(z,T(z))P_2(z,T(z))=0\quad\hbox{on }W_0.
$$
Since $W_0$ is connected, by the identity theorem, one factor vanishes identically on $W_0$. Thus $I(\Gamma)$ is a prime ideal. Let
\begin{equation}\label{eq:graph-projection}
\pi:\mathcal V\longrightarrow\C^n,\qquad\pi(z,w)=z,
\end{equation}
and let $f_i$ and $g_j$ be the coordinate projections onto the target blocks. More precisely,
write $w=(w^{(1)},\cdots,w^{(m+v)})$, where $w^{(i)}\in\C^{N_i}$ for $1\leq i\leq m$ and $w^{(m+j)}\in\C^{M_j}$ for $1\leq j\leq v$.
The coordinate projections on $\mathcal V$ are
\begin{equation}\label{eq:ambient-coordinate-maps}
	\pi(z,w)=z;\qquad
	f_i(z,w)=w^{(i)}, ~1\leq i\leq m; \qquad
	g_j(z,w)=w^{(m+j)},~~1\leq j\leq v.
\end{equation}
The image of $\pi$ contains $W_0$, so $d:=\dim\mathcal V\geq n$. Our goal is to prove $d=n$.

\begin{definition}\label{defn:regular}
For an affine algebraic variety $Z\subseteq\C^a$, a function is regular if it is locally, in the Zariski topology, a quotient of polynomials with denominator nowhere zero. Equivalently, it is the restriction of a polynomial on $\C^a$; see \cite[Chapter~I, Theorem~3.2]{Ha}. We write $\cO(Z)$ for the ring of regular functions and
$$
\cO(Z)^\times:=\{q\in\cO(Z):1/q\in\cO(Z)\}.
$$
A map between affine algebraic varieties is regular if its coordinate functions are regular. An algebraic isomorphism is a bijective regular map whose inverse is regular.
\end{definition}

In particular, $\pi$, $f_i$, and $g_j$ are regular. Let $\mathcal V^*$ be the conjugate algebraic variety, the Zariski closure of the graph $\Gamma^*$ of $\overline T$. On $\mathcal V\times\mathcal V^*$ define
\begin{equation}\label{eq:ambient-kernels}
	\begin{aligned}
		\mathcal D(x,y)&:=\rho_D(\pi(x),\overline\pi(y)),\\
		\mathcal R_i(x,y)&:=\rho_i(f_i(x),\overline{f_i}(y)),\\
		\mathcal S_j(x,y)&:=\sigma_j(g_j(x),\overline{g_j}(y)).
	\end{aligned}
\end{equation}
All these functions are regular and are not identically zero, since they equal $1$ at the pair of graph points over $(0,0)$.

A key ingredient in proving algebraicity of the map is Lemma~\ref{lem:early-lift} below. Its proof uses
the following consequence of \eqref{eq:graph-closure}: every nonempty open subset  of $\Gamma$ (respectively, $\Gamma^*$) in the usual Euclidean topology is Zariski dense in $\mathcal V$ (respectively, $\mathcal V^*$). Indeed, any polynomial vanishing on such a piece of $\Gamma$ vanishes
on all of $\Gamma$ by the identity theorem, since the domain of the graph is connected. It therefore belongs to $I(\Gamma)$ and hence vanishes on $\mathcal V$. The same argument
applies to $\Gamma^*$. Consequently, if a rational function on $\mathcal V$ (whose denominator
does not vanish identically on $\Gamma$) vanishes on a nonempty open piece of $\Gamma$, then it is identically zero on $\mathcal V$. The same assertion holds for the conjugate graph $\Gamma^*$ and its Zariski closure $\mathcal V^*$.

\begin{lemma}\label{lem:early-lift}
The following identity holds as an identity of rational one-forms on $\mathcal V\times\mathcal V^*$:
\begin{equation}\label{eq:ambient-log}
\frac{d\mathcal D}{\mathcal D}=\sum_{i=1}^m\lambda_i\frac{d\mathcal R_i}{\mathcal R_i}
-\sum_{j=1}^v\mu_j\frac{d\mathcal S_j}{\mathcal S_j}.
\end{equation}
\end{lemma}

\begin{proof}
First differentiate \eqref{eq:polarized-log} with respect to $\zeta_\alpha$, $1\leq\alpha\leq n$. After differentiation, fix $\zeta$ to be a point in $W_0$. This gives an identity of rational functions in $(z,T(z))$. Indeed, the values of $\overline{F_i}$ and $\overline{G_j}$ at $\zeta$, as well as their first derivatives there, are constants in this identity, and all generic norms are polynomials. The identity therefore extends from $\Gamma$ to $\mathcal V$ by Zariski density. Consequently,
\begin{equation}\label{eq:one-sided-log-derivatives}
	\begin{aligned}
		\partial_{\zeta_\alpha}\log\rho_D(\pi(x),\zeta)
		={}&\sum_i\lambda_i\partial_{\zeta_\alpha}\log\rho_i(f_i(x),\overline{F_i}(\zeta))\\
		&-\sum_j\mu_j\partial_{\zeta_\alpha}\log\sigma_j(g_j(x),\overline{G_j}(\zeta)), \quad 1 \leq \alpha \leq n.
	\end{aligned}
\end{equation}
This holds for $x$ in the intersection $\mathcal A$ of $\mathcal V$ with a sufficiently small ambient neighborhood of the graph point over $0$, and for $\zeta$ in a small ball centered at $0$. Choose these neighborhoods so that all the generic norms appearing in \eqref{eq:one-sided-log-derivatives} are close to $1$. By \eqref{eq:one-sided-log-derivatives}, the difference between the two sides of \eqref{eq:one-sided-lift} has zero gradient with respect to $\zeta$, and hence is constant in $\zeta$. Choosing the holomorphic logarithms to vanish at $\zeta=0$, where all the generic norms equal $1$, we see that this constant is zero. Thus
\begin{equation}\label{eq:one-sided-lift}
	\begin{aligned}
		\log\rho_D(\pi(x),\zeta)
		={}&\sum_i\lambda_i\log\rho_i(f_i(x),\overline{F_i}(\zeta))\\
		&-\sum_j\mu_j\log\sigma_j(g_j(x),\overline{G_j}(\zeta)).
	\end{aligned}
\end{equation}

Let $x_0$ be a smooth point in this neighborhood  $\mathcal A$ and let $\xi\in T_{x_0}\mathcal V$. Take the differential of \eqref{eq:one-sided-lift} with respect to the $x$-variable near $x_0$,
and evaluate it at $x_0$ on $\xi$. This yields an identity of rational functions in $(\zeta, \overline T(\zeta))$, the coordinates of the conjugate graph.  The values of $\pi,f_i,g_j$ and of their derivatives in direction $\xi$ at the fixed point $x_0$ are merely coefficients. Zariski density of the conjugate graph extends this identity to $y\in\mathcal V^*$.  Since this holds for every smooth point $x_0\in\mathcal A$ and every $\xi\in T_{x_0}\mathcal V$, we obtain
$$
\frac{d_x\mathcal D}{\mathcal D}
=\sum_i\lambda_i\frac{d_x\mathcal R_i}{\mathcal R_i}
-\sum_j\mu_j\frac{d_x\mathcal S_j}{\mathcal S_j}
$$
on $\Reg(\mathcal A)\times\Reg(\mathcal V^*)$ wherever the denominators
are nonzero. Here $d_x$ denotes differentiation with respect to $x$,
with $y$ fixed.

Applying the same argument with the two variables interchanged gives
the corresponding identity for $d_y$ on $\Reg(\mathcal V)\times\Reg(\mathcal A^*)$ wherever the denominators are nonzero. Adding these identities and using $d=d_x+d_y$, we obtain
\eqref{eq:ambient-log} on the nonempty open subset of $\Reg(\mathcal A)\times\Reg(\mathcal A^*)$ where all denominators are nonzero. Since $\mathcal V\times\mathcal V^*$ is irreducible,
this equality of rational one-forms holds identically.
\end{proof}

\begin{remark}\label{rem:lift-caution}
In the proof of Lemma~\ref{lem:early-lift}, it would not suffice merely to check that the two sides
of \eqref{eq:ambient-log} have the same pullback under the inclusion
$\Gamma\times\Gamma^*\hookrightarrow\mathcal V\times\mathcal V^*$.
The graph may have smaller dimension than its Zariski closure, and pullback only tests tangent
directions to the graph. For example, under the map $s(z)=(z,e^z)$, the nonzero algebraic one-form $dw-w\,dz$ on $\C^2$ has pullback
	$$s^*(dw-w\,dz)=d(e^z)-e^z\,dz=0,$$
even though the graph is Zariski dense in $\C^2$. 
\end{remark}

\subsection{The product-unit theorem}

In this subsection, we present a special case of the Rosenlicht product-unit theorem to prepare for later proofs. Before that, we recall the following definition.

\begin{definition}\label{defn:vanishing order}
Let $M$ be a (connected) complex manifold of dimension $n$, and let $E\subseteq M$ be an irreducible analytic hypersurface (so that its regular part $\mathrm{Reg}(E)$ is connected). Let $f$ be a nonzero meromorphic function on $M$. Choose any point $a\in \mathrm{Reg}(E)$ such that no irreducible zero or pole hypersurface of $f$, other than $E$ itself, passes through $a$. There exist local holomorphic coordinates
$(u,\eta)\in\C\times\C^{n-1}$ centered at $a$ such that $E=\{u=0\}$ locally. After shrinking the coordinate neighborhood, we can write
	$$f=u^k\psi,$$
where $k\in\mathbb Z$ and $\psi$ is holomorphic and nowhere zero. By the connectedness of $\mathrm{Reg}(E),$ the integer $k$ is independent of the choice of the point $a$. It is also independent of the local coordinates, and the local defining function $u$. It is called the order of $f$ along $E$ and is denoted by $\ord_E(f)$. A positive order indicates a zero along $E$, while a negative order indicates a pole.
\end{definition}

A general form of the Rosenlicht product-unit theorem was proved by Conrad \cite[Theorem~1.1]{ConradUnits}. In this paper, we need only the following special case for smooth irreducible affine algebraic varieties. For the reader's convenience and to make the paper self-contained, we include a proof using complex analytic arguments on compact complex manifolds.

\begin{lemma}[A special case of Rosenlicht's product-unit theorem; see Conrad {\cite[Theorem~1.1]{ConradUnits}}] \label{lem:RC}
Let $X$ and $Y$ be nonempty smooth irreducible complex affine algebraic varieties. If $Q\in\cO(X\times Y)^\times$, then
\begin{equation}\label{eq:Conrad-factor}
		Q(x,y)=a(x)b(y),\qquad
		a\in\cO(X)^\times,\quad b\in\cO(Y)^\times.
\end{equation}
The factors are unique up to replacing $(a,b)$ by $(ca,c^{-1}b)$ for a nonzero constant $c$. In particular, for any $x_0\in X$ and $y_0\in Y$,
\begin{equation}\label{eq:rectangle-identity}
		Q(x,y)Q(x_0,y_0)=Q(x,y_0)Q(x_0,y).
\end{equation}
\end{lemma}

\begin{proof}
If either variety is a point, there is nothing to prove. Otherwise, let $X^c$ and $Y^c$ be projective closures of $X$ and $Y$, respectively. By resolution of singularities, there exist smooth projective varieties $\overline X$ and $\overline Y$ together with projective birational morphisms
$$\mu_X:\overline X\longrightarrow X^c,\qquad\mu_Y:\overline Y\longrightarrow Y^c,$$
which are isomorphisms over $X$ and $Y$, respectively. Identifying $X$ and $Y$ with their inverse images, the spaces $\overline X$ and $\overline Y$ are compact connected complex manifolds, and $X$ and $Y$ are Zariski-open subsets of $\overline X$ and $\overline Y$, respectively. Moreover, the rational function $Q$ extends meromorphically to $\overline X\times\overline Y$.

Let $D_1,\cdots,D_r$ be the irreducible hypersurface components of $\overline X\setminus X$, and let $E_1,\cdots,E_s$ be those of $\overline Y\setminus Y$. Since $Q$ and $1/Q$ are regular on $X\times Y$, every zero or pole hypersurface of the extended function is contained in
	$$
	\bigl((\overline X\setminus X)\times\overline Y\bigr)
	\cup\bigl(\overline X\times(\overline Y\setminus Y)\bigr).
	$$
By irreducibility and dimension considerations, every irreducible hypersurface in the support of $\operatorname{div}(Q)$ is of the form $D_j\times\overline Y$ or $\overline X\times E_l$. Therefore,
\begin{equation}\label{eq:vertical-horizontal-divisor}
		\operatorname{div}(Q)
		=\sum_{j=1}^r m_j(D_j\times\overline Y)
		+\sum_{l=1}^s n_l(\overline X\times E_l),
\end{equation}
where $m_j=\ord_{D_j\times\overline Y}(Q)$ and $n_l=\ord_{\overline X\times E_l}(Q).$ Here a positive coefficient represents a zero of $Q$, while a negative coefficient represents a pole. Fix $x_0\in X$ and $y_0\in Y$, and set
$$a(x):=Q(x,y_0),\qquad b(y):=Q(x_0,y).$$
Since $Q\in\cO(X\times Y)^\times$, we have $a\in\cO(X)^\times$ and $b\in\cO(Y)^\times.$
Regarded as rational functions, $a$ and $b$ extend meromorphically to $\overline X$ and $\overline Y$, respectively. Note that every zero or pole hypersurface of the extended functions
$a$ and $b$ is contained in $\overline X\setminus X$ and $\overline Y\setminus Y$, respectively.  We claim that
\begin{equation}\label{eq:slice-orders}
\operatorname{div}_{\overline X}(a)=\sum_{j=1}^r m_jD_j,\qquad \operatorname{div}_{\overline Y}(b)=\sum_{l=1}^s n_lE_l.
\end{equation}
We verify the first identity. Fix $j$ and choose a smooth point $p\in D_j$ lying on no other irreducible hypersurface component of $\overline X\setminus X$. Let $u$ be a local holomorphic coordinate defining $D_j$ near $p$. Since $y_0\in Y$, no hypersurface $\overline X\times E_l$ passes through $(p,y_0)$. Thus, by \eqref{eq:vertical-horizontal-divisor}, the only possible zero or pole hypersurface of $Q$ near $(p,y_0)$ is $D_j\times\overline Y$, along which $Q$ has order $m_j$. By Definition~\ref{defn:vanishing order}, we therefore have the local factorization
$$Q(x,y)=u(x)^{m_j}v(x,y),$$
where $v$ is holomorphic and nowhere zero near $(p,y_0)$. Restricting to $y=y_0$, we obtain
$$a(x)=u(x)^{m_j}v(x,y_0),$$
and hence $\ord_{D_j}(a)=m_j$. Since $a$ has neither zeros nor poles on $X$, and the $D_j$ exhaust the irreducible hypersurfaces contained in $\overline X\setminus X$, this proves the first identity in \eqref{eq:slice-orders}.
The second follows in the same way by restricting to $x=x_0$. It follows from \eqref{eq:vertical-horizontal-divisor} and \eqref{eq:slice-orders} that
\begin{equation}\label{eq:H-definition}
H(x,y):=\frac{Q(x,y)Q(x_0,y_0)}{a(x)b(y)}
\end{equation}
has neither zeros nor poles along any hypersurface of $\overline X\times\overline Y$. By the same argument used at the end of the proof of Proposition~\ref{prpn:Q-unit}, since $\overline X\times\overline Y$ is smooth, it follows that $H$ is regular (and hence holomorphic) on the compact connected complex manifold $\overline X\times\overline Y$. By the maximum principle, $H$ is constant. Its value at $(x_0,y_0)$ is one, and thus $H \equiv 1.$ Then \eqref{eq:H-definition} yields \eqref{eq:rectangle-identity} and, after absorbing the nonzero constant \(Q(x_0,y_0)^{-1}\) into \(b\), the factorization \eqref{eq:Conrad-factor}.
	
Finally, if $a(x)b(y)=\widehat a(x)\widehat b(y)$, then $\widehat a(x)/a(x)=b(y)/\widehat b(y)$. Fixing either variable shows that both sides equal a nonzero constant. This proves the uniqueness assertion. 
\end{proof}

\subsection{Algebraicity of the maps}\label{subsec:algebraicity}

We are ready to prove the algebraicity of the maps in Conjecture~\ref{conj1}.

\begin{theorem}\label{thm:algebraicity}
Assume the hypotheses of Conjecture~\ref{conj1}, with the dimension assumption weakened to $n\geq1$.
Then $\dim\mathcal V=n$; and every component of $F_i$ and $G_j$, $1\leq i\leq m, 1\leq j\leq v$, is algebraic.
\end{theorem}

\begin{proof}
Suppose that $\dim\mathcal V=n+r$ with $r>0$. Choose a graph point
$$p=(z_0,T(z_0))\in\Reg(\mathcal V).$$
Such a point exists arbitrarily close to the graph point over $0$: the singular locus is a proper algebraic subset, whereas every nonempty open piece of the graph is Zariski dense in $\mathcal V$.

We claim the projection $\pi$ in \eqref{eq:graph-projection} is a submersion at $p$. Indeed, the graph map $s(z)=(z,T(z))$ maps a small neighborhood of $z_0$ in $\mathbb{C}^n$ into $\Reg(\mathcal V)$ and satisfies $\pi\circ s=\operatorname{id}$. Therefore
\begin{equation}\label{eq:submersion-graph}
d\pi_p\circ ds_{z_0}=\operatorname{id}_{\C^n}.
\end{equation}
Thus $d\pi_p$ is surjective.

By the holomorphic implicit function theorem, the fiber $\pi^{-1}(z_0)$ is smooth
near $p$ and has dimension $r$. Let $A_0$ be the irreducible component of the algebraic fiber
$\pi^{-1}(z_0)$ containing $p$. The fiber is smooth near $p$, so $p$ is a smooth point of $A_0$ and lies on no other component of $\pi^{-1}(z_0)$.

The singular points of $A_0$, the points of $A_0\setminus\Reg(\mathcal V)$, and the intersections of $A_0$ with the other components of $\pi^{-1}(z_0)$ form a closed algebraic subset of $A_0$ that does not contain $p$. Since $A_0$ is affine, we can choose a regular function $a$ on
$A_0$ that vanishes on this subset but satisfies $a(p)\neq0$. Set
$$A:=\{q\in A_0:a(q)\neq0\}.$$
This is a Zariski-open neighborhood of $p$ in $A_0$. It is also affine, since it is algebraically isomorphic to
$$\{(q,t)\in A_0\times\C:a(q)t=1\}.$$
Consequently, $A$ is a smooth irreducible affine variety of dimension $r>0$.

Pull back \eqref{eq:ambient-log} to $A\times A^*$. Note $\pi(x)=z_0$ for all $x \in A$, and $\overline{\pi}(y)=\overline{z_0}$ for all $y \in A^*$. Thus the pullback of $\mathcal D$ is the positive constant
$$\mathcal D(x,y)=\rho_D(z_0,\overline{z_0})>0.$$
Consequently, its logarithmic differential along $A\times A^*$ vanishes. The pullbacks of
$\mathcal R_i$ and $\mathcal S_j$ are regular functions on $A\times A^*$ and are not identically zero: at $(p,\overline{p})$, their values are the positive generic norms of the corresponding
target points. Using the same notation for these pullbacks, we obtain
\begin{equation}\label{eq:fiber-log}
0=\sum_i\lambda_i\frac{d\mathcal R_i}{\mathcal R_i} -\sum_j\mu_j\frac{d\mathcal S_j}{\mathcal S_j}
\end{equation}
wherever the denominators are nonzero.

Suppose that one of the functions $\mathcal R_i$ or $\mathcal S_j$ has an irreducible zero hypersurface $E\subseteq A\times A^*$. Set
$$
a_i:=\ord_E(\mathcal R_i),\qquad
b_j:=\ord_E(\mathcal S_j).
$$
These are nonnegative integers, and at least one is positive. Choose a smooth point $\hat{p}\in E$ that does not lie on any irreducible zero hypersurface of any $\mathcal R_i$ or $\mathcal S_j$
in $A\times A^*$ distinct from $E$. Since $A\times A^*$ is a smooth complex manifold of dimension $2r$, there are local holomorphic coordinates
$$
(u,v_1,\ldots,v_{2r-1})
$$
centered at $\hat{p}$ such that $E=\{u=0\}$. After shrinking the coordinate neighborhood, we may write
\begin{equation}\label{eq:local-factorization-fiber-E}
	\mathcal R_i=u^{a_i}h_i,\qquad
	\mathcal S_j=u^{b_j}k_j,
\end{equation}
where all $h_i$ and $k_j$ are holomorphic and nowhere zero.

Let $\gamma:(\mathbb D_\epsilon,0)\to(A\times A^*,\hat{p})$ be the holomorphic disk transverse to $E$ given in the above coordinates by
$$
\gamma(\tau)=(\tau,0,\ldots,0),\qquad
\mathbb D_\epsilon:=\{\tau\in\C:|\tau|<\epsilon\},
$$
for sufficiently small $\epsilon>0$. Set
$$
r_i(\tau):=\mathcal R_i(\gamma(\tau)),\qquad
s_j(\tau):=\mathcal S_j(\gamma(\tau)).
$$
These functions are holomorphic on $\mathbb D_\epsilon$ and nowhere zero on the punctured disk. Pulling
\eqref{eq:fiber-log} back to the punctured disk gives
\begin{equation}\label{eq:disk-fiber-log}
0=\sum_i\lambda_i\frac{r_i'(\tau)}{r_i(\tau)}\,d\tau
-\sum_j\mu_j\frac{s_j'(\tau)}{s_j(\tau)}\,d\tau,
\qquad 0\neq\tau\in\mathbb D_\epsilon.
\end{equation}
By \eqref{eq:local-factorization-fiber-E}, we have
$$r_i(\tau)=\tau^{a_i}h_i(\gamma(\tau)),\qquad s_j(\tau)=\tau^{b_j}k_j(\gamma(\tau)),$$
where $h_i(\gamma(0))$ and $k_j(\gamma(0))$ are nonzero.
Consequently,
$$
\operatorname{Res}_{\tau=0}
\left(\frac{r_i'(\tau)}{r_i(\tau)}\,d\tau\right)=a_i,
\qquad
\operatorname{Res}_{\tau=0}
\left(\frac{s_j'(\tau)}{s_j(\tau)}\,d\tau\right)=b_j.
$$
Taking residues in \eqref{eq:disk-fiber-log} yields
$$0=\sum_i\lambda_i a_i-\sum_j\mu_j b_j.$$
By the noncancellation condition \eqref{eq:noncancellation},
all $a_i$ and $b_j$ must be zero, a contradiction. Thus none
of the functions $\mathcal R_i$ or $\mathcal S_j$ has a zero
hypersurface on $A\times A^*$.

A nonempty zero set of a nonzero holomorphic function on a
complex manifold has a hypersurface component. Hence all
$\mathcal R_i$ and $\mathcal S_j$ are nowhere zero on
$A\times A^*$. Their reciprocals are regular, since locally
they are quotients with nonvanishing denominators. Therefore
$$
\mathcal R_i,\mathcal S_j\in\cO(A\times A^*)^\times.
$$
By Lemma~\ref{lem:RC}, there are regular units $a_i,c_j$ on $A$ and $b_i,d_j$ on $A^*$ such that
\begin{equation}\label{eq:fiber-separation}
	\mathcal R_i(x,y)=a_i(x)b_i(y),\qquad
	\mathcal S_j(x,y)=c_j(x)d_j(y)~~\text{on}~~A\times A^*.
\end{equation}
Since $f_i(p)=F_i(z_0)\in\Omega_i$ and $\Omega_i$ is open, $f_i$ takes values in $\Omega_i$ near $p$ in $A$. On the real diagonal $y=\overline{x}$, \eqref{eq:fiber-separation} gives, for $x \in A$ near $p$,
$$\rho_i(f_i(x),\overline{f_i(x)}) =a_i(x)b_i(\overline{x})>0.$$
Here $a_i(x)$ is nowhere zero and holomorphic, while $b_i(\overline{x})$ is nowhere zero and antiholomorphic. Taking local logarithms shows that the logarithm of the left side is
pluriharmonic. Hence the pullback of the metric $\omega_{\Omega_i}$ to $A:$
$$
(f_i|_A)^*\omega_{\Omega_i}
=-\ii\partial\bar\partial
\log\rho_i(f_i(x),\overline{f_i(x)})=0
\quad\hbox{near }p.
$$
Since $g_{\Omega_i}$ is positive definite, $df_i$ vanishes on $T_pA$. The same argument applied to $\mathcal S_j$ shows that $dg_j$ vanishes on $T_pA$ for every $j$. Also $d\pi$
vanishes on $T_pA$, since $\pi$ is constant on $A$. But $\pi,f_1,\cdots,f_m,g_1,\cdots,g_v$ together give the ambient coordinates of $\mathcal V\subseteq\C^{n+N}$. Since $A$ is a smooth submanifold embedded in $\C^{n+N}$ and $\dim A=r>0$, the differentials of the ambient coordinate functions cannot
all vanish on $T_pA$. This is a contradiction.

Consequently, $\dim\mathcal V=n$. Let $\C(\mathcal V)$ denote the field of rational functions on the irreducible algebraic variety $\mathcal V$. The transcendence degree of $\C(\mathcal V)$ over
$\C$, denoted by $\operatorname{trdeg}_{\C}\C(\mathcal V)$, is the largest number of elements of $\C(\mathcal V)$ that are algebraically independent over $\C$. By \cite[Chapter~I, Theorem~3.2(d)]{Ha},
$$
\operatorname{trdeg}_{\C}\C(\mathcal V)=\dim\mathcal V=n.
$$
Since $\pi(\mathcal V)$ contains $W_0$, the coordinate functions $z_1,\cdots,z_n$ are algebraically independent in $\C(\mathcal V)$. Thus $\C(\mathcal V)$ is algebraic over $\C(z_1,\cdots,z_n)$. In particular, every target coordinate, i.e., each component of $f_i$ and $g_j$, satisfies a nontrivial polynomial equation over $\C(z_1,\cdots,z_n)$. Clearing denominators and restricting these equations to $\Gamma$ gives polynomial equations for the components of $T$. By the identity theorem, these equations hold on the original connected domain $V$. This proves the algebraicity of $F_i$ and $G_j$.
\end{proof}

\begin{remark}\label{rem:source-general-algebraicity}
It is clear that the preceding argument does not require the condition $n\geq2$, and hence the conclusion of Theorem~\ref{thm:algebraicity} also holds when $n=1.$
\end{remark}

\subsection{The algebraic covers}\label{subsec:covers}

We now assume $n\geq2$ and retain the notation from \S\ref{subsec:graph}. In particular, $\Gamma$ and $\mathcal V$ are defined in \eqref{eq:graph} and \eqref{eq:graph-closure}, respectively, and the coordinate projections $\pi,f_i,g_j$ are as in \eqref{eq:ambient-coordinate-maps}.
By Theorem~\ref{thm:algebraicity}, $\dim\mathcal V=n$, and every component of each $F_i$ and $G_j$ is algebraic. Consequently, there exists a nonzero polynomial $h\in\C[z]$ such that, setting
\begin{equation}\label{eq:A-U}
A:=\{z\in\C^n:h(z)=0\},\qquad U:=\C^n\setminus A,
\end{equation}
the restriction $\pi:\pi^{-1}(U)\to U$ is a finite-sheeted unbranched holomorphic covering.
In particular, $\pi^{-1}(U)$ is smooth of dimension $n$.

We realize $\pi^{-1}(U)$ as a closed affine algebraic variety by introducing an additional coordinate $\mu$ and setting
\begin{equation}\label{eq:X-affine}
	X:=\{(z,w,\mu)\in\mathcal V\times\C:h(z)\mu=1\}.
\end{equation}
The defining equations of $\mathcal V$, together with $h(z)\mu=1$, show that $X$ is Zariski closed in
$\C^n\times\C^N\times\C$. The map
$$(z,w)\longmapsto\left(z,w,\frac1{h(z)}\right)$$
is an algebraic isomorphism from $\pi^{-1}(U)$ onto $X$. Thus $X$ is a smooth irreducible affine variety of dimension $n$. Its projection
$$
\pi_X:X\longrightarrow U,\qquad \pi_X(z,w,\mu)=z,
$$
is a finite-sheeted holomorphic covering and hence is locally biholomorphic. The projection $(z,w,\mu)\mapsto(z,w)$ identifies $X$ with $\pi^{-1}(U)\subseteq\mathcal V$. Recall the target coordinate maps $f_i$ and $g_j$ on $\mathcal V$ from \eqref{eq:ambient-coordinate-maps}.
Under the identification $X\cong\pi^{-1}(U)$, we continue to denote their restrictions to $X$ by $f_i$
and $g_j$.

Let $\mathcal V^*:=\{(\zeta,\eta) \in \C^n \times \C^N:(\overline{\zeta},\overline{\eta}) \in\mathcal V\}$ be the conjugate variety of $\mathcal V$, and let $\overline\pi:\mathcal V^*\to\C^n$ be the coordinate projection $\overline\pi(\zeta,\eta)=\zeta$. Similarly as above, define
$$
X^*:=\{(\zeta,\eta,\nu)\in\mathcal V^*\times\C:
\overline h(\zeta)\nu=1\},\qquad
U^*:=\{\zeta\in\C^n:\overline h(\zeta)\neq0\}.
$$
Then $X^*$ is the conjugate variety of $X$. It is a smooth irreducible affine variety, and the projection
$$ \overline{\pi_X}:X^*\longrightarrow U^*,\qquad \overline{\pi_X}(\zeta,\eta,\nu)=\zeta,
$$
is a finite-sheeted unbranched covering. The projection
$(\zeta,\eta,\nu)\mapsto(\zeta,\eta)$ identifies $X^*$ with the Zariski-open subset
$(\overline\pi)^{-1}(U^*)\subseteq\mathcal V^*$.

Write $\eta=(\eta^{(1)},\ldots,\eta^{(m+v)})$, where $\eta^{(i)}\in\C^{N_i}$ and $\eta^{(m+j)}\in\C^{M_j}$. Define the coordinate projections on $\mathcal V^*$ by
$$
\overline{f_i}(\zeta,\eta)=\eta^{(i)},\qquad
\overline{g_j}(\zeta,\eta)=\eta^{(m+j)}.
$$
Under the above identification, we continue to denote their restrictions to $X^*$ by $\overline{f_i}$ and $\overline{g_j}$. These are precisely the conjugate maps of $f_i$ and $g_j$: for $y\in X^*$,
$$
\overline{f_i}(y)=\overline{f_i(\overline y)},\qquad
\overline{g_j}(y)=\overline{g_j(\overline y)}.
$$

Let $W:=W_0\cap U$ and $W^*:=W_0^*\cap U^*$. These are nonempty dense open
connected subsets of $W_0$ and $W_0^*$, respectively. The original branches of the maps define holomorphic sections
\begin{equation}\label{eq:sections}
	s_X(z):=\left(z,T(z),\frac1{h(z)}\right),\qquad
	\overline{s_X}(\zeta):=\left(\zeta,\overline{T}(\zeta),
	\frac1{\overline h(\zeta)}\right),
	\quad z\in W,\ \zeta\in W^*.
\end{equation}
They satisfy $\pi_X\circ s_X=\operatorname{id}_W$,
$\overline{\pi_X}\circ\overline{s_X}=\operatorname{id}_{W^*}$, and
$$
f_i\circ s_X=F_i,\quad g_j\circ s_X=G_j,\quad
\overline{f_i}\circ\overline{s_X}=\overline{F_i},\quad
\overline{g_j}\circ\overline{s_X}=\overline{G_j}.
$$
Since $\pi_X$ and $\overline{\pi_X}$ are locally biholomorphic, the images $s_X(W)$ and $\overline{s_X}(W^*)$ are open subsets of $X$ and $X^*$, respectively, in the usual Euclidean topology.

Regard $x\in X$ and $y\in X^*$ as independent variables. Pull back the functions defined in \eqref{eq:ambient-kernels} to $X\times X^*$, and retain the notation
$\mathcal D,\mathcal R_i,\mathcal S_j$. More explicitly, for $(x, y) \in X\times X^*,$
\begin{align}
	\mathcal D(x,y)&=\rho_D(\pi_X(x),\overline{\pi_X}(y)),\label{eq:delta-tilde}\\
	\mathcal R_i(x,y)&=\rho_i(f_i(x),\overline{f_i}(y)),\label{eq:R-tilde}\\
	\mathcal S_j(x,y)&=\sigma_j(g_j(x), \overline{g_j}(y)).\label{eq:S-tilde}
\end{align}
Their pullbacks by $s_X\times\overline{s_X}$ are $\Delta,R_i,S_j$, respectively. In particular, these functions are not identically zero, since they are nowhere zero on $s_X(W)\times\overline{s_X}(W^*)$ by the choice of $W_0$ in \S\ref{subsec:potentials}. Since the maps $X\to\mathcal V$ and $X^* \to\mathcal V^*$ identify $X$ and $X^*$ with smooth open subsets of these varieties, Lemma~\ref{lem:early-lift} gives
\begin{equation}\label{eq:log-differential}
	\frac{d\mathcal D}{\mathcal D}
	=\sum_i\lambda_i\frac{d\mathcal R_i}{\mathcal R_i}
	-\sum_j\mu_j\frac{d\mathcal S_j}{\mathcal S_j}
\end{equation}
as an identity of rational one-forms on $X\times X^*$.

\section{The quotient functions are regular units}\label{sec:incidence}

Set
\begin{equation}\label{eq:Sigma}
\Sigma:=\{(x,y)\in X\times X^*:\mathcal D(x,y)=0\}.
\end{equation}
In contrast to the ball case, $\Sigma$ can be singular when $D$ is an irreducible bounded symmetric domain of higher rank (see Remark~\ref{rem:singular-source}). The purpose of this section is twofold: to show that $\Sigma$ is a nonempty irreducible hypersurface and that $\mathcal D$ has order one along it, and to prove that the quotients of $\mathcal R_i$ and $\mathcal S_j$ by appropriate powers of $\mathcal D$ are regular units on $X\times X^*$. The proof uses the Segre geometry of Hermitian symmetric spaces developed in \cite{FHX}. 

\subsection{The compact dual and its Segre geometry}\label{subsec:segre}

In this subsection, we recall the Segre geometry of Hermitian symmetric spaces of compact type
from \cite{FHX}. Let $M$ be the Hermitian symmetric space of compact type that is the
compact dual of $D$. Via the Borel embedding (cf.~\cite{Wolf} and \cite[Section~2.1]{CM}), we regard $D$ as an open subset of the standard affine cell $\C^n\subseteq M$. Equip $M$ with the canonical minimal projective embedding $\iota:M\hookrightarrow\PP^L$ described in \cite[Section~2.2]{FHX}. Here $L$ is a positive integer depending on $M.$ Following the explicit realization of the canonical embedding in \cite[\S2.2]{FHX}, its restriction to the aforementioned affine cell $\C_z^n$ can be written as
\begin{equation}\label{eq:compact-dual-embedding}
	\iota(z)=[1:P_1(z):\cdots:P_L(z)],
\end{equation}
where each $P_\alpha$ is a homogeneous polynomial in $z$ with real coefficients, although their degrees may differ. On this cell $\C_z^n,$ the pullback of the Fubini--Study form from $\PP^L$ to $M$ has the K\"ahler potential $\log\rho_D^+(z,\overline z),$ where
$$
\rho_D^+(z,\overline z):=1+\sum_{\alpha=1}^L|P_\alpha(z)|^2.
$$
We use $\rho_D^+(z,\zeta)$ for the polarization of this polynomial. In the
standard coordinates, it is related to the generic norm of $D$ by the following (see \cite{Loos1977,CM,FHX}):
\begin{equation}\label{eq:compact-dual-kernel}
	\rho_D^+(z,\zeta)=\rho_D(-z,\zeta)=\rho_D(z,-\zeta).
\end{equation}

We now recall the definition of Segre varieties from \cite[Section~2.1]{FHX}. Regard $M$ as a projective variety of $\PP^L$ under the canonical minimal projective embedding $\iota$ as above.
Let $M^*:=\{[\overline p]:[p]\in M\}$ be the conjugate projective variety of $M$, where $[\overline p]$ denotes coordinatewise conjugation of the homogeneous coordinates of $[p]$. (We indeed have $M^*=M$, but still write it as $M^*$ to distinguish the two copies). Let $\overline\iota:M^*\to\PP^L$ denote the corresponding minimal canonical embedding of $M^*$. For $b\in M$, write
$\iota(b)=[\iota_0(b):\cdots:\iota_L(b)]$. Likewise, for $q\in M^*$, write
$\overline\iota(q)=[\overline\iota_0(q):\cdots:\overline\iota_L(q)]$.

For $q\in M^*$, define the hyperplane in $\PP^L$:
$$
H_q:=\left\{[T_0:\cdots:T_L]\in\PP^L:
\sum_{\alpha=0}^L\overline\iota_\alpha(q)T_\alpha=0\right\}.
$$
This definition is independent of the choice of homogeneous
coordinates representing $\overline\iota(q)$.
The Segre variety in $M$ associated with $q$ is defined by
$$ Q_q:=\iota^{-1}(H_q)
=\left\{b\in M:\sum_{\alpha=0}^L
\iota_\alpha(b)\overline\iota_\alpha(q)=0\right\}.
$$
Thus, identifying $M$ with its image under $\iota$, we have $Q_q=M\cap H_q$.
If $q=\zeta$ lies in the affine cell $\C_\zeta^n$ of $M^*$, then
$$
Q_\zeta\cap\C_z^n
=\{z\in\C_z^n:\rho_D^+(z,\zeta)=0\}.
$$
Make the linear change of coordinates $z\mapsto-z$ on the affine cell of $M$, and continue to denote the new coordinate by $z$. With this convention,
\eqref{eq:compact-dual-kernel} gives
\begin{equation}\label{eq:affine-segre}
	Q_\zeta\cap\C_z^n
	=\{z\in\C_z^n:\rho_D(z,\zeta)=0\}.
\end{equation}
Thus, for each fixed $\zeta\neq0$, the affine hypersurface $\rho_D(z,\zeta)=0$ is the affine part of the projective Segre variety $Q_\zeta$.

When $\zeta=0$, the affine intersection in \eqref{eq:affine-segre} is empty. In fact, $Q_0$ is the hyperplane section at infinity, namely, the complement of the affine cell of $M$. Each $Q_q$, $q\in M^*$, is an irreducible hyperplane section of $M$, and $Q_0$ is a Schubert variety. Moreover, let $K$ denote the identity component of the holomorphic isometry group of $M$. Then $K$ acts transitively on the family of Segre varieties. These facts are explained in \cite[Sections~2.1 and 7]{FHX}.

Let $\mathcal G:=\operatorname{Aut}^0(M)$ be the identity component of the holomorphic automorphism group of $M$. It is the complexification of $K$, and is a connected complex algebraic
group acting transitively on $M$. Moreover, $K$ is a Zariski-dense compact subgroup of $\mathcal G$. For the application of Kleiman's theorem below (see \S\ref{subsec:lifted-segre}), we will use $\mathcal G$ instead of $K$. The action of $\mathcal G$ extends to projective linear
transformations of $\PP^L$ under the canonical minimal embedding $\iota$. More precisely, each $g\in\mathcal G$ is represented by an invertible $(L+1)\times(L+1)$ matrix $A$, defined up to
a nonzero scalar, such that
$$
\iota(g(b))=[AZ]\qquad\text{when }\iota(b)=[Z].
$$
For such $A$, define the dual projective transformation by
$$
[\Xi]\longmapsto[A^{-T}\Xi].
$$
We claim that this transformation preserves $\overline{\iota}(M^*)$. Indeed, when $g\in K$, we may choose $A$ unitary, in which case $A^{-T}=\overline A$. Conjugating the action of $A$ on $\iota(M)$ shows that $\overline A$ preserves $\overline{\iota}(M^*)$.
Then the Zariski density of $K$ in $\mathcal G$ gives the claim
for every $g\in\mathcal G$. Thus the dual projective transformation induces an automorphism of $M^*$,
which we denote by $g^\vee$.

For homogeneous coordinate vectors $Z,\Xi\in\C^{L+1}$, we have
$$(AZ)^T(A^{-T}\Xi)=Z^T\Xi.$$
The definition of the Segre varieties therefore gives $g(Q_q)=Q_{g^\vee(q)}$. By \cite[Section~2.1]{FHX}, the Segre varieties form a single orbit under $K$, and hence also under $\mathcal G$. In particular, for every $q\in M^*$, there exists $g\in\mathcal G$ such that $Q_q=g(Q_0)$.

The classical normality theorem for Schubert varieties states that every Schubert variety in a complex flag variety is normal; see \cite[Section~2.1, Theorem~2.1.1]{Brion}. Since $Q_0$ is a Schubert
variety and every $Q_q$ is its translate, every $Q_q$ is normal. Consequently, the singular locus of $Q_q$ has codimension at least two in $Q_q$.

By \cite[Propositions~7.2 and 7.4]{FHX}, together with \eqref{eq:compact-dual-kernel}, $\rho_D(z,\zeta)$ is irreducible as a polynomial in $(z,\zeta)$. Moreover, for every $\zeta\neq0$, the
polynomial $\rho_D(\cdot,\zeta)$ is nonconstant and irreducible, with constant term $\rho_D(0,\zeta)=1$. We emphasize that these statements concern hypersurfaces in the base spaces $\C^n\times\C^n$ and $\C^n$. Their irreducibility, however, does not imply that their inverse images
under the finite coverings are irreducible. We prove the required irreducibility in the next subsections.

\begin{remark}\label{rem:singular-source}
When $D$ is the unit ball $\B^n$, we have $\rho_D(z,\zeta)=1-z\cdot\zeta$. In this case, the hypersurface $H:=\{\rho_D(z,\zeta)=0\}$ in $U\times U^*$ is a smooth affine variety.
Since $\pi_X\times \overline{\pi_X}$ is locally biholomorphic, $\Sigma$ is also smooth. For a general irreducible bounded symmetric domain $D$, however, the irreducibility of $\rho_D(z,\zeta)$ does not guarantee smoothness of its zero hypersurface $H$. Singular points lying in $U\times U^*$ lift
to singular points of $\Sigma$. For example, take $D$ to be the classical domain of type I in $M_{2\times2}(\C)$, and write $z,\zeta\in M_{2\times2}(\C)$. Its polarized generic norm is
$$
\rho_D(z,\zeta)=\det(I_2-z\zeta^{\mathsf T}).
$$
At $z=z_0$ and $\zeta=(z_0^{-1})^{\mathsf T}$, for any $z_0\in\mathrm{GL}_2(\C)$, this polynomial and all its first derivatives vanish. Indeed, writing $z=z_0+A$ and $\zeta=(z_0^{-1})^{\mathsf T}+B$ gives
$$
\rho_D(z,\zeta)
=\det(z_0B^{\mathsf T}+Az_0^{-1}+AB^{\mathsf T}),
$$
which has no terms of degree less than two in the entries of $A$ and $B$. These singularities cannot all be avoided by shrinking $U$ to a smaller nonempty Zariski-open subset.
\end{remark}

In the following, we identify $\C_z^n$ with the affine chart of $M$ used in \eqref{eq:affine-segre} and hence regard $U\subseteq\C_z^n$ as an open subset
of $M$. Let $\pi_X:X\to U$ be the finite covering map defined in
\S\ref{subsec:covers}. Composing $\pi_X$ with the inclusion
$U\hookrightarrow M$, we also regard $\pi_X$ as a map into $M$. We similarly
identify $\C_\zeta^n$ with the standard affine chart of $M^*$ and regard
$U^*\subseteq\C_\zeta^n$ as an open subset of $M^*$. Since $U=\C^n\setminus\{h=0\}$, the complement $M\setminus U$ is the
union of the hyperplane section at infinity $Q_0$ and the closure in $M$ of
$\{h=0\}$. In particular, $M\setminus U$ is an algebraic hypersurface
in $M$.

\subsection{Irreducibility after lifting a general Segre variety}
\label{subsec:lifted-segre}

For $y\in X^*$, put $\zeta=\overline{\pi_X}(y)$ and define
$$\Sigma_y:=\{x\in X:\rho_D(\pi_X(x),\zeta)=0\}.$$
We want to show that $\Sigma_y$ is nonempty and irreducible for general
$y\in X^*$. When $D$ is a ball, $M=\PP^n$, and its Segre varieties
comprise all projective hyperplanes. After compactifying $X$, one obtains
a surjective regular map from a smooth projective variety to $\PP^n$, and
the classical Bertini irreducibility theorem (see \cite[Theorem~3.3.1]{La}) can be applied to prove the irreducibility of $\Sigma_y$; see Xiao \cite{Xiao2026} and Yuan--Zhang \cite{YZ2026}.
For a general irreducible bounded symmetric domain and its dual, however, the hyperplanes $H_q$, $q\in M^*$, defining the Segre varieties $Q_q$ may form a lower-dimensional family in the full dual projective space. Bertini's theorem only gives irreducibility for a general hyperplane in the full
dual projective space, and this lower-dimensional family could lie entirely in the exceptional set. A different argument is therefore needed. The next two lemmas instead use the Fulton--Lazarsfeld connectedness theorem \cite{FL,La}, together with Kleiman's theorem \cite{Kleiman} (cf.~\cite[Lemma in Section~B.9.2]{Fu}), to establish connectedness and normality of the inverse images of general translates of $Q_0$. These two properties then imply irreducibility.

\begin{lemma}\label{lem:projective-segre-pullback}
Let $Z$ be a smooth irreducible projective variety of dimension $n\geq2$, and let $f:Z\to M$ be a surjective regular map. For $q$ in a nonempty Zariski-open subset of $M^*$, the inverse image $f^{-1}(Q_q)$ is a nonempty, connected, normal, and irreducible hypersurface of $Z$.
\end{lemma}

\begin{proof}
First, $f^{-1}(Q_q)$ is nonempty by surjectivity of $f$. Locally, it is defined by pulling back a defining equation of $Q_q$. Since $f$ is surjective and $Z$ is irreducible, the pulled-back equation is not identically zero, so $f^{-1}(Q_q)$ is a hypersurface of $Z$.

We use the connectedness theorem for hyperplane pullbacks \cite[Theorem~3.3.3]{La}: if $Z$ is
irreducible and projective and $\varphi:Z\to\PP^L$ has image dimension greater than the codimension
of a linear subspace $H$, then $\varphi^{-1}(H)$ is connected. Apply this theorem to
$\varphi=\iota\circ f$, where $\iota$ is the canonical minimal embedding of $M$,
and to the hyperplane $H_q$ defining $Q_q$. Since $\dim\varphi(Z)=n\geq2$, it follows that
$f^{-1}(Q_q)$ is connected for every $q\in M^*$.

We next prove normality of $f^{-1}(Q_q)$ for general translates $Q_q$ of $Q_0$, where $0 \in \C_\zeta^n \subseteq M^*$. Recall that $\mathcal G$ acts transitively
on $M$ and that every Segre variety $Q_q$ is a translate of $Q_0$. Let
$S=\operatorname{Sing}(Q_0)$. Since $Q_0$ is normal and has dimension $n-1$, we have
$\dim S\leq n-3$. We apply \cite[Theorem~2(ii), p.~290]{Kleiman} with the ambient variety $X=M$,
the variety $Y=\operatorname{Reg}(Q_0)$ equipped with its inclusion into $M$,
and our variety $Z$ equipped with the map $f:Z\to M$. Both $\operatorname{Reg}(Q_0)$
and $Z$ are smooth and irreducible. For $g\in\mathcal G$, the fiber product appearing in Kleiman's theorem is
$$\{(b,p)\in\operatorname{Reg}(Q_0)\times Z:g(b)=f(p)\}.$$
Projection onto $p$ identifies it with $f^{-1}(g(\operatorname{Reg}(Q_0)))$, since
$b=g^{-1}(f(p))$ is uniquely determined by $p$. Thus \cite[Theorem~2(ii), p.~290]{Kleiman}
gives a nonempty Zariski-open subset $\hat{\mathcal G}\subseteq\mathcal G$ such that, for every
$g\in\hat{\mathcal G}$, the set
$$
f^{-1}(g(\operatorname{Reg}(Q_0)))
=f^{-1}(g(Q_0))\setminus f^{-1}(g(S))
$$
is smooth.

The algebraic variety $S$ has finitely many irreducible components, which we denote by
$S_1,\cdots,S_t$. For each $1\leq\alpha\leq t$, apply \cite[Theorem~2(i), p.~290]{Kleiman}
with $X=M$, $Y=S_\alpha$ equipped with its inclusion into $M$, and the same map
$f:Z\to M$. This gives a nonempty Zariski-open subset $\mathcal G_\alpha\subseteq\mathcal G$
such that, for every $g\in\mathcal G_\alpha$, the inverse image $f^{-1}(g(S_\alpha))$
is either empty or has dimension
$$
\dim f^{-1}(g(S_\alpha))
=\dim S_\alpha+\dim Z-\dim M
=\dim S_\alpha\leq n-3.
$$
Set
$$
\mathcal G_0:=\hat{\mathcal G}\cap\bigcap_{\alpha=1}^t\mathcal G_\alpha.
$$
Since $\mathcal G$ is irreducible, this finite intersection is a nonempty Zariski-open subset
of $\mathcal G$. For every $g\in\mathcal G_0$, the hypersurface $f^{-1}(g(Q_0))$ is smooth
away from $f^{-1}(g(S))$, and
$$
f^{-1}(g(S))=\bigcup_{\alpha=1}^t f^{-1}(g(S_\alpha))
$$
has dimension at most $n-3$. Here, if $S$ is empty, we take $t=0$ and
$\mathcal G_0=\hat{\mathcal G}$.

Since $f^{-1}(g(Q_0))$ has dimension $n-1$, its singular locus, contained in
$f^{-1}(g(S))$, has codimension at least two. We now use the following consequence of
Serre's normality criterion: a hypersurface in a smooth complex algebraic variety is normal
if its singular locus has codimension at least two in the hypersurface; see
\cite[Chapter II, Proposition 8.23]{Ha}.
Applying this criterion to the hypersurface $f^{-1}(g(Q_0))$ in the smooth variety $Z$
proves that it is normal for every $g\in\mathcal G_0$.

We now pass from $g\in\mathcal G$ to the parameter $q\in M^*$. Consider the orbit map
$$
\theta:\mathcal G\longrightarrow M^*,\qquad
\theta(g):=g^\vee(0).
$$
Here $g^\vee$ is as defined in \S\ref{subsec:segre}. By the discussion there, $g(Q_0)=Q_{\theta(g)}$.

The map $\theta$ is algebraic, since the dual action $g^\vee$ is given in homogeneous coordinates by $[\Xi]\mapsto[A^{-T}\Xi]$. It is also surjective, as the automorphisms $\{g^\vee: g\in\mathcal G\}$ act transitively on $M^*$. Since $\mathcal G_0$ is Zariski dense in $\mathcal G$,
the image $\theta(\mathcal G_0)$ is Zariski dense in $M^*$.
By Chevalley's theorem \cite[Chapter~II, Exercise~3.19]{Ha},
the dense image $\theta(\mathcal G_0)$ is constructible. Consequently, it contains a nonempty Zariski-open subset $B\subseteq M^*$ (see \cite[Chapter~II, Exercise~3.18(b)]{Ha}). For every $q\in B$,
there exists $g\in\mathcal G_0$ such that $Q_q=g(Q_0)$, so $f^{-1}(Q_q)$ is normal.

We have already shown that $f^{-1}(Q_q)$ is nonempty and connected for every $q\in M^*$.
For $q\in B$, it is also normal. Since a nonempty connected normal algebraic
variety is irreducible, $f^{-1}(Q_q)$ is irreducible for every $q\in B$.
\end{proof}

Let $\pi_X:X\to U\subseteq M$ be the finite covering map constructed in \S\ref{subsec:covers}. Here $U$ is regarded as an open subset of $M$, as described at the end of \S\ref{subsec:segre}.

\begin{lemma}\label{lem:general-segre-fiber}
There exists a nonempty Zariski-open subset $G\subseteq\C^n$ such that, for every $\zeta\in G$, the set
$$
\{x\in X:\rho_D(\pi_X(x),\zeta)=0\}
$$
is a nonempty irreducible algebraic hypersurface of $X$.
\end{lemma}

\begin{proof}
Choose a projective closure $Y$ of $X$, and let
$$
\Gamma_{\pi_X}:=\{(x,\pi_X(x)):x\in X\}\subseteq Y\times M.
$$
Let $J$ be the Zariski closure of $\Gamma_{\pi_X}$ in $Y\times M$. Then $J$ is irreducible
and projective, and $\Gamma_{\pi_X}\cong X$ is a smooth Zariski-open subset of $J$.

By resolution of singularities, 
there exist a smooth irreducible projective variety $\widehat J$
and a birational morphism
$r:\widehat J\longrightarrow J$
that is an isomorphism over $\Gamma_{\pi_X}$. We identify $X$ with the
Zariski-open subset $r^{-1}(\Gamma_{\pi_X})$ of $\widehat J$. Define
$$
f:=\operatorname{pr}_2\circ r:\widehat J\longrightarrow M,
$$
where $\operatorname{pr}_2:J\to M$ is the projection onto the second factor. Under this
identification, $f|_X=\pi_X$. Since $\widehat J$ is projective, the image of $f$ is closed in $M$.
It contains the dense open subset $U$, so $f$ is surjective.
Moreover, $\dim\widehat J=n$. Applying
Lemma~\ref{lem:projective-segre-pullback} with $Z=\widehat J$ therefore
shows that $f^{-1}(Q_q)$ is a nonempty irreducible hypersurface of
$\widehat J$ for $q$ in a nonempty Zariski-open subset $G_1$ of $M^*$. Intersecting this subset $G_1$ with
the affine chart $\C_\zeta^n\subseteq M^*$, we obtain the same conclusion for $f^{-1}(Q_\zeta)$ for
$\zeta$ in a nonempty Zariski-open subset $G_2$ of $\C^n$.

We next assert that $f^{-1}(Q_\zeta)\cap X$ is nonempty for general $\zeta$. Since $f(X)=U$, this assertion is
equivalent to $Q_\zeta\cap U\neq\varnothing$. If $Q_\zeta\cap U$ is empty, then the irreducible
hypersurface $Q_\zeta$ is contained in $M\setminus U$ and hence equals one of the finitely many
hypersurface components of $M\setminus U$. We pause to prove the following claim.

{\bf Claim.} Distinct affine parameters $\zeta,\zeta'\in\C^n$ give distinct Segre varieties $Q_\zeta,Q_{\zeta'}$.

\medskip

{\bf Proof of Claim.}
First suppose that $Q_\zeta=Q_{\zeta'}$ with
$\zeta,\zeta'\neq0$. Their intersections with the affine cell have the
same zero set, so the irreducible polynomials
$\rho_D(z,\zeta)$ and $\rho_D(z,\zeta')$ differ by a nonzero constant
factor. Since both equal $1$ at $z=0$, they are equal. Using
\eqref{eq:compact-dual-kernel} and replacing $z$ by $-z$, we obtain
$\rho_D^+(z,\zeta)=\rho_D^+(z,\zeta').$	Subtracting the two sides and using the following formula
$$
\rho_D^+(z,\zeta)=1+\sum_{\alpha=1}^L
P_\alpha(z)\overline{P_\alpha}(\zeta),
$$
we obtain
$$
\sum_{\alpha=1}^L P_\alpha(z)
\bigl(\overline{P_\alpha}(\zeta)
-\overline{P_\alpha}(\zeta')\bigr)=0.
$$
If any coefficient
$\overline{P_\alpha}(\zeta)-\overline{P_\alpha}(\zeta')$
were nonzero, the equation
$$
\sum_{\alpha=1}^L
\bigl(\overline{P_\alpha}(\zeta)
-\overline{P_\alpha}(\zeta')\bigr)T_\alpha=0
$$
would define a projective hyperplane in $\PP^L$ containing the affine cell
$\{[1:P_1(z):\cdots:P_L(z)]:z\in\C^n\}$ and hence its closure $\iota(M)$.
This contradicts the construction of the canonical minimal embedding $\iota$ in \cite{FHX}, whose image $\iota(M)$ is not contained in any hyperplane in $\PP^L$. Therefore
$\overline{P_\alpha}(\zeta)=\overline{P_\alpha}(\zeta')$
for every $\alpha$, so
$\overline\iota(\zeta)=\overline\iota(\zeta')$. Since
$\overline\iota$ is injective, we conclude that $\zeta=\zeta'$.

Moreover, $Q_0$ has empty intersection with the affine cell, whereas $Q_\zeta$ has nonempty intersection for $\zeta\neq0$. Thus $Q_0$ is distinct from all such $Q_\zeta$. This proves the claim. \qed

\bigskip

By the above claim, each hypersurface component
of $M\setminus U$ can equal $Q_\zeta$ for at most one parameter $\zeta$, so only finitely many parameters satisfy $Q_\zeta\cap U=\varnothing$. Removing these parameters from the preceding nonempty Zariski-open parameter set $G_2$, we obtain a nonempty Zariski-open subset $G\subseteq\C^n$ such that, for every $\zeta\in G$,
$$
\{x\in X:\rho_D(\pi_X(x),\zeta)=0\}
=f^{-1}(Q_\zeta)\cap X
$$
is a nonempty Zariski-open subset of the irreducible hypersurface $f^{-1}(Q_\zeta)$. It is therefore a nonempty irreducible algebraic hypersurface of $X$.
\end{proof}

\subsection{The entire lifted hypersurface}\label{subsec:whole-Sigma}

Let $\pr_{X^*}:X\times X^*\to X^*$ denote the projection onto the second factor. Consider its restriction $\mathcal P:=\pr_{X^*}|_\Sigma:\Sigma\to X^*$. Since $\Sigma$ may be singular (see Remark~\ref{rem:singular-source}), we cannot expect $\mathcal P$ to be a holomorphic submersion everywhere.

\begin{lemma}\label{lem:Sigma-irreducible}
The set $\Sigma$ is a nonempty irreducible hypersurface of $X\times X^*$, and
$\ord_\Sigma(\mathcal D)=1$. 
\end{lemma}

\begin{proof}
Let $\mathcal P$ be the projection defined above. For each $y\in X^*$, set
$$
\Sigma_y:=\{x\in X:
\rho_D(\pi_X(x),\overline{\pi_X}(y))=0\}.
$$
Then $\mathcal P^{-1}(y)=\Sigma_y\times\{y\}$. By Lemma~\ref{lem:general-segre-fiber}, there exists a nonempty Zariski-open subset $G\subseteq\C^n$ such that whenever
$\overline{\pi_X}(y)\in G$, the set $\Sigma_y$ is a nonempty irreducible algebraic hypersurface of $X$. Hence, in this case, the fiber
$\mathcal P^{-1}(y)$ is also nonempty and irreducible, of dimension $n-1$. Since $U^*$ and $G$ are nonempty Zariski-open subsets of $\C^n$, their intersection is nonempty. Since
$\overline{\pi_X}:X^*\to U^*$ is surjective,
$$
X_G^*:=(\overline{\pi_X})^{-1}(G\cap U^*)
$$
is a nonempty Zariski-open subset of $X^*$. For every $y\in X_G^*$, we have
$\overline{\pi_X}(y)\in G$, and thus the fiber $\mathcal P^{-1}(y)=\Sigma_y\times\{y\}$ is nonempty. Consequently, $\Sigma$ is nonempty.
	
Recall the function $\mathcal D$ defined in \eqref{eq:delta-tilde} is not identically zero. Therefore, every irreducible component $C$ of its zero set $\Sigma$
in the smooth $2n$-dimensional variety $X\times X^*$ has dimension $2n-1$. Let $\overline{\mathcal P(C)}$ denote the Zariski closure of
$\mathcal P(C)$ in $X^*$, and consider the map $\mathcal P|_C:C\longrightarrow\overline{\mathcal P(C)}.$
For each $y\in X^*$, put $\zeta=\overline{\pi_X}(y)$. Then
$$
\mathcal P^{-1}(y)
=\pi_X^{-1}\bigl(\{z\in U:\rho_D(z,\zeta)=0\}\bigr)\times\{y\}.
$$
Since $\rho_D(0,\zeta)=1$, the zero set inside the parentheses is either empty or a hypersurface of dimension $n-1$. As $\pi_X$ is locally biholomorphic, every nonempty fiber $\mathcal P^{-1}(y)$ has dimension $n-1$. Every fiber of $\mathcal P|_C$ is contained in a fiber of $\mathcal P$ and therefore has dimension at most
$n-1$. Applying the fiber dimension theorem (see \cite[Chapter~II, Exercise~3.22(b)]{Ha}) to the map
$\mathcal P|_C:C\to\overline{\mathcal P(C)}$, whose image is dense, gives
$$
\dim C-\dim\overline{\mathcal P(C)}\leq n-1.
$$
Consequently,
$$
\dim\overline{\mathcal P(C)}
\geq\dim C-(n-1)=n.
$$
Since $X^*$ is irreducible of dimension $n$, we have
$\dim\overline{\mathcal P(C)}=n$ and $\overline{\mathcal P(C)}=X^*$. Thus every irreducible component of $\Sigma$ projects densely onto $X^*$.

Suppose that $\Sigma$ has distinct irreducible components $C_1,\cdots,C_q$, with $q\geq2$.
Each $C_a$ has dimension $2n-1$ and projects densely onto $X^*$. For each $C_a$, the generic
fiber dimension theorem (see \cite[Chapter~II, Exercise~3.22(e)]{Ha}) gives a nonempty Zariski-open subset $V_a\subseteq X^*$ such that, for every $y\in V_a$, the fiber
$C_a\cap\mathcal P^{-1}(y)$ is a nonempty algebraic subvariety of $\mathcal P^{-1}(y)$ of dimension $n-1$.
	
For $a\neq b$, set $J_{ab}:=C_a\cap C_b$. Since $C_a$ and $C_b$ are distinct irreducible
hypersurfaces, $\dim J_{ab}\leq2n-2$. Consider an irreducible component $\widehat J_{ab}$
of $J_{ab}$. If its image under $\mathcal P$ is not dense in $X^*$, then
$\mathcal P^{-1}(y)\cap\widehat J_{ab}$ is empty for $y$ outside the proper Zariski-closed subset
$\overline{\mathcal P(\widehat J_{ab})}$. If its image is dense, the generic fiber dimension theorem shows that its fibers have dimension at most $(2n-2)-n=n-2$ over a nonempty Zariski-open subset of $X^*$. Since $J_{ab}$ has finitely many irreducible components, intersecting these open subsets gives a nonempty Zariski-open subset $V_{ab}\subseteq X^*$ such that, for every $y\in V_{ab}$, the intersection $\mathcal P^{-1}(y)\cap J_{ab}$ has dimension
at most $n-2$.
	
Since $X^*$ is irreducible, the intersection
$$
X_G^*\cap\bigcap_{a=1}^q V_a
\cap\bigcap_{1\leq a<b\leq q}V_{ab}
$$
is nonempty. Choose $y$ in this intersection. Then every
$C_a\cap\mathcal P^{-1}(y)$ is nonempty and has dimension
$n-1$, while every $J_{ab}\cap\mathcal P^{-1}(y)$ has dimension at most $n-2$. Thus the algebraic subvarieties $C_a\cap\mathcal P^{-1}(y)$, $1\leq a\leq q$, of $\mathcal P^{-1}(y)$ are distinct and have dimension $n-1$. But $\mathcal P^{-1}(y)=\Sigma_y\times\{y\}$ is irreducible of dimension $n-1$, so it cannot contain two distinct subvarieties of dimension $n-1$. This contradiction proves that $\Sigma$ is irreducible.

Finally, since $\rho_D(z,\zeta)$ is an irreducible polynomial, its gradient cannot vanish identically on its zero set. (Otherwise, $\rho_D$ would divide each of its partial derivatives, which is impossible.) Hence there is a point of
$\{\rho_D=0\}\cap(U\times U^*)$ at which $d\rho_D\neq0$. Since
$\pi_X\times\overline{\pi_X}$ is surjective onto $U\times U^*$ and locally biholomorphic, the chain rule gives a point of $\Sigma$ at which $d\mathcal D\neq0$. Near this point, $\mathcal D$ is a local holomorphic coordinate and vanishes to order one. Since $\Sigma$ is irreducible, we conclude that $\ord_\Sigma(\mathcal D)=1$.
\end{proof}

\subsection{Orders of the lifted functions and regular units}\label{subsec:units}

Define the integers
\begin{equation}\label{eq:orders-Sigma}
	k_i:=\ord_\Sigma(\mathcal R_i),\qquad
	\ell_j:=\ord_\Sigma(\mathcal S_j).
\end{equation}
These orders are finite because none of these functions is identically zero, and they are nonnegative because the functions are regular. Thus $k_i,\ell_j\in\mathbb Z_{\geq0}$. We also recall that the ambient variety $X\times X^*$ is smooth.

\begin{proposition}\label{prpn:Q-unit}
The quotients
\begin{equation}\label{eq:def-Q}
    Q_i:=\frac{\mathcal R_i}{\mathcal D^{k_i}},\qquad
    P_j:=\frac{\mathcal S_j}{\mathcal D^{\ell_j}},
    \quad 1\leq i\leq m,\quad 1\leq j\leq v,
\end{equation}
and their reciprocals extend to regular functions on $X\times X^*$. In particular,
\begin{equation}\label{eq:Q-unit}
    Q_i,P_j\in\cO(X\times X^*)^\times.
\end{equation}
Moreover,
\begin{equation}\label{eq:residue-Sigma}
    1=\sum_{i=1}^m\lambda_i k_i-\sum_{j=1}^v\mu_j\ell_j.
\end{equation}
\end{proposition}

\begin{proof}
We first show that none of the functions $\mathcal R_i$ and $\mathcal S_j$ has an irreducible zero hypersurface other than $\Sigma$. Suppose, to the contrary, that one of these functions vanishes along an irreducible hypersurface $E\neq\Sigma$ in $X\times X^*$. Set
$$
a_i:=\ord_E(\mathcal R_i),\qquad b_j:=\ord_E(\mathcal S_j).
$$
Since these functions are regular, all the orders are nonnegative integers, and at least one is positive. Choose a smooth point $\hat{p}$ of $E$ such that $\hat{p}\notin\Sigma$ and no irreducible zero hypersurface of any $\mathcal R_i$ or $\mathcal S_j$, other than $E$ itself, passes through $\hat{p}$. There exist local holomorphic coordinates
$(u,v_1,\ldots,v_{2n-1})$
on $X\times X^*$, centered at $\hat{p}$, such that $E=\{u=0\}$ locally. After shrinking the coordinate neighborhood, we may write
\begin{equation}\label{eq:local-factorization-E}
    \mathcal D=c,\qquad
    \mathcal R_i=u^{a_i}c_i,\qquad
    \mathcal S_j=u^{b_j}d_j,
\end{equation}
where $c$, all the $c_i$, and all the $d_j$ are holomorphic and nowhere zero in the neighborhood. Let $\gamma:(\mathbb D_\epsilon,0)\to(X\times X^*,\hat{p})$ be the holomorphic disk transverse to $E$ at $\hat{p}$ given, in these coordinates, by
$$
\gamma(\tau)=(\tau,0,\ldots,0),\qquad
\tau\in\mathbb D_\epsilon:=\{\tau\in\C:|\tau|<\epsilon\},
$$
where $\epsilon>0$ is sufficiently small. Put
$$
\delta(\tau):=\mathcal D(\gamma(\tau)),\qquad
r_i(\tau):=\mathcal R_i(\gamma(\tau)),\qquad
s_j(\tau):=\mathcal S_j(\gamma(\tau)).
$$
These are holomorphic functions on $\mathbb D_\epsilon$. By \eqref{eq:local-factorization-E},
$\delta$ is nowhere zero there, while $r_i$ and $s_j$ are nowhere zero on the punctured disk.
Pulling back \eqref{eq:log-differential} to this punctured disk gives
\begin{equation}\label{eq:disk-log-E}
    \frac{\delta'(\tau)}{\delta(\tau)}\,d\tau
    =\sum_{i=1}^m\lambda_i\frac{r_i'(\tau)}{r_i(\tau)}\,d\tau
    -\sum_{j=1}^v\mu_j\frac{s_j'(\tau)}{s_j(\tau)}\,d\tau,
    \quad 0\neq\tau\in\mathbb D_\epsilon.
\end{equation}
Since $\delta(0)\neq0$, the one-form on the left has residue zero at $\tau=0$. By \eqref{eq:local-factorization-E} and \eqref{eq:disk-log-E}, a residue calculation similar to
the one following \eqref{eq:disk-fiber-log} yields	
	
$$0=\sum_{i=1}^m\lambda_i a_i-\sum_{j=1}^v\mu_j b_j.$$
Condition \eqref{eq:noncancellation} implies that all the orders $a_i$ and $b_j$ are zero, contradicting the fact that at least one is positive. Thus none of the functions $\mathcal R_i$ and $\mathcal S_j$ has an irreducible zero hypersurface other than $\Sigma$.
	
We next prove \eqref{eq:residue-Sigma}. By Lemma~\ref{lem:Sigma-irreducible},
$\ord_\Sigma(\mathcal D)=1$. Choose a general smooth point of $\Sigma$ and local holomorphic coordinates $(u,v_1,\ldots,v_{2n-1})$ centered at this point, with $\Sigma=\{u=0\}$ locally.
After shrinking the coordinate neighborhood, we may write
$$
\mathcal D=u \hat{c},\qquad
\mathcal R_i=u^{k_i} \hat{c_i},\qquad
\mathcal S_j=u^{\ell_j}\hat{d_j},
$$
where $\hat{c}, \hat{c_i}$ and $\hat{d_j}$ are all holomorphic and nowhere zero. Pulling back \eqref{eq:log-differential} to the punctured transverse disk $\gamma(\tau)=(\tau,0,\ldots,0)$ and taking residues at $\tau=0$, as above, yields \eqref{eq:residue-Sigma}.
	
Finally, each $Q_i$ and $P_j$ has neither zeros nor poles along any irreducible hypersurface of $X\times X^*$. Indeed, along $\Sigma$, the orders of the numerator and denominator of $Q_i$ are both $k_i$, and those of $P_j$ are both $\ell_j$. Away from $\Sigma$, none of their numerators or denominators has a zero hypersurface. The same conclusion holds for their reciprocals. Since $X\times X^*$ is smooth, a rational function having no poles along any irreducible hypersurface is regular; see, e.g., \cite[Chapter~3, pp.~149--150]{Sh}. Applying this to $Q_i$, $P_j$, and their reciprocals, we conclude that all these functions are regular, and \eqref{eq:Q-unit} follows immediately.
\end{proof}

Although it is not needed for algebraicity, the dimension restriction $n\geq2$ is essential for componentwise rigidity; see Remark~\ref{rmk1d3}. Furthermore,
to make our proofs more transparent to readers, in Remark~\ref{rem:square-root} below, we use Mok's square-root map \cite{Mok2012} to illustrate why the above arguments
fail when $n=1$.

\begin{remark}\label{rem:square-root}

To describe Mok's map explicitly, it is convenient to use the Poincar\'e upper half-plane model $\mathbb H$, which is biholomorphic to the unit disk. In this model, Mok's square-root map \cite{Mok2012} is
$$F(\tau)=\bigl(\sqrt{\tau},\ii\sqrt{\tau}\bigr),\qquad \tau\in\mathbb H,$$
where $\sqrt{\tau}$ denotes the branch satisfying $0<\arg\sqrt{\tau}<\pi/2$. The map $F$ is a
holomorphic isometric embedding from $\mathbb H$ to $\mathbb H^2$, although neither component
is a holomorphic isometry up to a positive constant factor. The Zariski closure of its graph is
$$\mathcal V=\{(\tau,w_1,w_2)\in\C^3:w_1^2=\tau,\ w_2=\ii w_1\}.$$
Writing $s=w_1$, we parametrize $\mathcal V$ by $s\mapsto(s^2,s,\ii s)$. In this parametrization, the projection $\pi:\mathcal V\to\C$ is $s\mapsto s^2$ and is ramified at $s=0$. We may therefore take $h(\tau)=\tau$ and $U=\C^*$. The affine variety constructed in \eqref{eq:X-affine} is
$$X=\{(\tau,w_1,w_2,\mu)\in\C^4:w_1^2=\tau,\ w_2=\ii w_1,\ \tau\mu=1\}.$$
It is algebraically isomorphic to $\C^*$ via the parametrization
$s\mapsto(s^2,s,\ii s,1/s^2)$. Similarly, the conjugate cover is
$$ X^*=\{(\sigma,\eta_1,\eta_2,\nu)\in\C^4: \eta_1^2=\sigma,\ \eta_2=-\ii\eta_1,\ \sigma\nu=1\}, $$
parametrized by $t\mapsto(t^2,t,-\ii t,1/t^2)$, with $t\in\C^*$. Thus, on $X\times X^*$,
we have $\tau=s^2$ and $\sigma=t^2$, with $s,t\neq0$.
	
The complexified boundary equation $\rho_{\mathbb H}$ of $\mathbb H$ is $\tau-\sigma=0$. Under the above parametrizations, its inverse image in $X\times X^*$ is
$$ \Sigma=\{s^2-t^2=0\}=\Sigma_+\sqcup\Sigma_-, \qquad \Sigma_+:=\{s=t\},\quad \Sigma_-:=\{s=-t\},\quad s,t\neq0. $$
The two components are disjoint because $s,t\neq0$. Thus $\Sigma$ is reducible and disconnected, so the irreducibility assertion of Lemma~\ref{lem:Sigma-irreducible} fails in this example.
	
In the upper half-plane model, the analogues of the polarized functions in \eqref{eq:delta-tilde} and \eqref{eq:R-tilde} are given, up to nonzero constant factors, by
$$ \mathcal D=s^2-t^2,\qquad \mathcal R_1=s-t,\qquad \mathcal R_2=s+t. $$
Hence $\mathcal R_1$ vanishes on $\Sigma_+$ and is nowhere zero on $\Sigma_-$, whereas
$\mathcal R_2$ vanishes on $\Sigma_-$ and is nowhere zero on $\Sigma_+$. In particular, neither function vanishes identically on all of $\Sigma$. Indeed, since $\Sigma$ is reducible, the order
$\ord_\Sigma(\mathcal R_i)$ is not defined by Definition~\ref{defn:vanishing order}.
Moreover, 
$$ \frac{\mathcal R_1}{\mathcal D}=\frac{1}{s+t},\qquad \frac{\mathcal R_2}{\mathcal D}=\frac{1}{s-t}. $$
These quotients have poles along $\Sigma_-$ and $\Sigma_+$, respectively. More generally, dividing either $\mathcal R_i$ by any nonnegative integer power of $\mathcal D$ cannot produce a regular unit. Thus the regular unit conclusion in Proposition~\ref{prpn:Q-unit} cannot be obtained in this example.
\end{remark}

\section{Proof of Theorem~\ref{thm:main}}\label{sec:product-unit}

We are now ready to prove Theorem~\ref{thm:main}.

\begin{proof}[Proof of Theorem~\ref{thm:main}]
Retain the normalization and notation of \S\ref{sec:section1} and \S\ref{sec:incidence}.
By Proposition~\ref{prpn:Q-unit}, the functions
$$
Q_i=\frac{\mathcal R_i}{\mathcal D^{k_i}},\qquad
P_j=\frac{\mathcal S_j}{\mathcal D^{\ell_j}},
\quad 1\leq i\leq m,\quad 1\leq j\leq v,
$$
belong to $\cO(X\times X^*)^\times$. By Lemma~\ref{lem:RC}, for all $x,x'\in X$ and $y,y'\in X^*$, we have
\begin{equation}\label{eq:Q-rectangle}
Q_i(x,y)Q_i(x',y')=Q_i(x,y')Q_i(x',y),\qquad 1\leq i\leq m;
\end{equation}
\begin{equation}\label{eq:P-rectangle}
P_j(x,y)P_j(x',y')=P_j(x,y')P_j(x',y),\qquad 1\leq j\leq v.
\end{equation}
On the neighborhood $W_0\times W_0^*$ of $(0,0)$, define
\begin{equation}\label{eq:def-local-quotients}
q_i(z,\zeta):=\frac{R_i(z,\zeta)}{\Delta(z,\zeta)^{k_i}},\qquad
p_j(z,\zeta):=\frac{S_j(z,\zeta)}{\Delta(z,\zeta)^{\ell_j}}.
\end{equation}
These functions are holomorphic and nowhere zero, since $\Delta$, $R_i$, and $S_j$ are nowhere zero on $W_0\times W_0^*$. Recall the sections $s_X$ and $\overline{s_X}$ defined in \eqref{eq:sections}. For $z\in W$ and $\zeta\in W^*$, the definitions give
$$
q_i(z,\zeta)=Q_i(s_X(z),\overline{s_X}(\zeta)),\qquad
p_j(z,\zeta)=P_j(s_X(z),\overline{s_X}(\zeta)).
$$
Then by \eqref{eq:Q-rectangle} and \eqref{eq:P-rectangle}, we obtain
\begin{equation}\label{eq:local-rectangle}
\begin{aligned}
q_i(z,\zeta)q_i(z',\zeta')&=q_i(z,\zeta')q_i(z',\zeta),\\
p_j(z,\zeta)p_j(z',\zeta')&=p_j(z,\zeta')p_j(z',\zeta)
\end{aligned}
\end{equation}
whenever $z,z'\in W$ and $\zeta,\zeta'\in W^*$. Both sides of these identities are holomorphic for all four variables in $W_0$. Since $W$ and $W^*$ are dense in $W_0$ and $W_0^*$ respectively, the identities hold for all $z,z'\in W_0$ and $\zeta,\zeta'\in W_0^*$.

By the normalization $F_i(0)=G_j(0)=0$ and the generic norm property \eqref{eq:generic-norm-axes},
$$
q_i(z,0)=q_i(0,\zeta)=q_i(0,0)=1,\qquad
p_j(z,0)=p_j(0,\zeta)=p_j(0,0)=1.
$$
Setting $z'=\zeta'=0$ in \eqref{eq:local-rectangle} gives
$$q_i(z,\zeta)=p_j(z,\zeta)=1\quad\hbox{on }W_0\times W_0^*.$$
Substituting these identities into \eqref{eq:def-local-quotients}, we obtain
\begin{equation}\label{eq:individual-kernel-identities}
\begin{aligned}
\rho_i(F_i(z), \overline{F_i}(\zeta))&=\rho_D(z,\zeta)^{k_i},\\
\sigma_j(G_j(z), \overline{G_j}(\zeta))&=\rho_D(z,\zeta)^{\ell_j}
\end{aligned}
\end{equation}
on $W_0\times W_0^*$. Setting $\zeta=\overline z$ and applying $-\ii\partial\bar\partial\log$ gives
$$
F_i^*\omega_{\Omega_i}=k_i\omega_D,\qquad
G_j^*\omega_{\Lambda_j}=\ell_j\omega_D
\quad\hbox{on }W_0.
$$
By real analyticity, they hold on the original connected domain $V$. If $k_i=0$, positive definiteness of the metrics implies $dF_i=0$, so $F_i$ is constant. This contradicts the assumption in Conjecture~\ref{conj1}. The same argument applies to $G_j$ and $\ell_j$. Therefore, $k_i$ and $\ell_j$ are positive integers.

Applying Mok's extension theorem \cite[Theorem~1.3.1]{Mok2012} for individual holomorphic isometries between bounded symmetric domains, we conclude each component $F_i$ and $G_j$
extends to a holomorphic isometric embedding on the whole source domain $D$. We retain the notation $F_i$ and $G_j$ for these extensions.
Setting $\zeta=\overline z$ in \eqref{eq:individual-kernel-identities} and using real analyticity, we obtain
$$
\rho_i(F_i(z),\overline{F_i(z)})=\rho_D(z,\overline z)^{k_i},\qquad
\sigma_j(G_j(z),\overline{G_j(z)})=\rho_D(z,\overline z)^{\ell_j}
\quad\hbox{on }D.
$$
These identities immediately give properness of every component. Applying $-\ii\partial\bar\partial\log$ gives the corresponding metric identities \eqref{eq:component-isometry}. Finally, substituting \eqref{eq:component-isometry} into \eqref{eq:metric-identity} gives \eqref{eq:sum-factors}.  This proves Conjecture~\ref{conj1}.
\end{proof}

\begin{remark}\label{rmk:higher-rank}
Under the hypotheses of Conjecture~\ref{conj1}, if $\rank(D)\geq2$, every component
extends to a totally geodesic holomorphic embedding of $D$ into its target domain.
Indeed, by Theorem~\ref{thm:main}, each component is individually a holomorphic isometry
up to a positive constant factor. Mok's total-geodesy theorem (\cite[Theorem~1.3.2(a)]{Mok2012}) therefore applies to each component.
\end{remark}

\begin{remark}\label{rmk:balls}
Under the hypotheses of Conjecture~\ref{conj1}, if $D=\B^n$ and every target factor is a complex unit ball, then every component extends to a totally geodesic holomorphic isometric embedding. (Indeed,  a holomorphic isometric map between complex unit balls is totally geodesic.) This conclusion was previously established in \cite{CHYZ}. It was also known earlier by \cite{YZ} when all conformal factors are positive.
\end{remark}

\bigskip

{\small\scshape Department of Mathematics, University of California at San Diego, La Jolla, CA 92093, USA\par}
{\small\itshape Email address:} \texttt{m3xiao@ucsd.edu}

\end{document}